\documentclass[pdftex,11pt, oneside, reqno]{amsart}

\usepackage[utf8]{inputenc}            
\usepackage{amsmath}
\usepackage{algorithm,algpseudocode}
\usepackage{amssymb}  
\usepackage{amsbsy} 
\usepackage{listings}
\usepackage{booktabs}
\usepackage{graphics}
\usepackage{bm}
\usepackage{prodint}
\usepackage[multiple]{footmisc}
\usepackage{longtable} 
\usepackage{algorithm}
\usepackage{algorithmicx}
\usepackage{algpseudocode}
\usepackage{enumerate}
\usepackage[shortlabels]{enumitem}
\usepackage{bigints}

\IfFileExists{dsfont.sty}{
	\usepackage{dsfont}
	\usepackage[usenames,dvipsnames]{color} 
	
}{
	
}
\allowdisplaybreaks

\usepackage[numbers]{natbib}

\usepackage{setspace} 
\usepackage{pdfpages}
\usepackage{hyperref} 
\hypersetup{
    unicode=false,          
    pdftoolbar=true,        
    pdfmenubar=true,        
    pdffitwindow=false,     
    pdfstartview={FitH},    
    pdftitle={Aggregate Markov models in life insurance: estimation via the EM algorithm},    
    pdfauthor={Jamaal Ahmad and Mogens Bladt},     
    pdfsubject={Actuarial Mathematics},   
    pdfnewwindow=true,      
    colorlinks=true,       
    linkcolor=Black, 
    citecolor=Black,        
    filecolor=Black,      
    urlcolor=Black 
}

\usepackage[font=small
,format=plain
,labelfont=bf,up
,textfont=it,up
,width=\columnwidth
]{caption}

\usepackage[a4paper]{geometry}  
\usepackage[english]{babel}

\usepackage[autolanguage,np]{numprint}

\usepackage{doi}
\usepackage{float}

\usepackage{fancyhdr}
\usepackage{amsthm}
\theoremstyle{plain}
\newtheorem{Theorem}{Theorem}
\newtheorem{theorem}[Theorem]{Theorem}

\newtheorem{lemma}[Theorem]{Lemma}

\newtheorem{corollary}[Theorem]{Corollary}

\numberwithin{table}{section}
\numberwithin{figure}{section}
\numberwithin{equation}{section}

\definecolor{mogens}{rgb}{0.0, 0.0, 0.0}

\definecolor{jamaal}{rgb}{.8, .1,.1}

\definecolor{darkblue}{rgb}{.2, 0.2,.8}
\definecolor{darkgreen}{rgb}{0,0.5,0.3}
\definecolor{darkred}{rgb}{.8, .1,.1}

\newcommand\demoo{\xqed{$\triangle$}}
\makeatletter
\newcommand*\bigcdot{\mathpalette\bigcdot@{.5}}
\newcommand*\bigcdot@[2]{\mathbin{\vcenter{\hbox{\scalebox{#2}{$\m@th#1\bullet$}}}}}
\makeatother

\newcommand{\N}{\mathbb{N}}

\newcommand{\e}{{\rm e}}

\theoremstyle{definition}

\newcommand\xqed[1]{%
  \leavevmode\unskip\penalty9999 \hbox{}\nobreak\hfill
  \quad\hbox{#1}}

\theoremstyle{remark}

\newtheorem{remark}[Theorem]{Remark}

\numberwithin{Theorem}{section}

\numberwithin{equation}{section}

\allowdisplaybreaks

\newcommand{\md}{{\, \mathrm{d} }}

\newcommand{\tuborg}[1]{\{ #1 \}}

\DeclareMathAlphabet{\mathpzc}{OMS}{pzc}{m}{it}
\newcommand{\J}{\mathpzc{J}}

\newcommand{\vect}[1]{\pmb{#1}}
\newcommand{\mat}[1]{\boldsymbol{\bm #1}}

\makeatletter

\newcommand{\Rmnum}[1]{\expandafter\@slowromancap\romannumeral #1@}
\makeatother

\usepackage{wasysym} 

\usepackage{tikz}
\usetikzlibrary{arrows,positioning,calc} 

\tikzset{
    >=stealth',
    punkt/.style={
           rectangle,
           rounded corners,
           draw=black, thick,
           text width=7em,
           minimum height=2em,
           text centered},
    punktl/.style={
           re
           tangle,
           rounded corners,
           draw=black, thick,
           
           text width=7em,
           minimum height=2em,
           text centered},
    pil/.style={
           ->,
           shorten <=4pt,
           shorten >=4pt,},
    pildotted/.style={
           ->,
           shorten <=4pt,
           shorten >=4pt,
  dotted,
  }
}

\usepackage[scr=boondoxo]{mathalfa}
\newcommand\iay{\mathscr{i}}
\newcommand\jay{\mathscr{j}}

\newcommand\sai{\mathscr{s}}

\allowdisplaybreaks

\usepackage{subfig}
\usepackage{todonotes}
\title[]{Aggregate Markov models in life insurance: estimation via the EM algorithm}
\author{Jamaal Ahmad}
\address{Department of Mathematical Sciences, University of Copenhagen, Universitetsparken 5, DK-2100 Copenhagen \O, Denmark.}
\email{jamaal@math.ku.dk}
\author{Mogens Bladt}
\address{Department of Mathematical Sciences, University of Copenhagen, Universitetsparken 5, DK-2100 Copenhagen \O, Denmark.}
\email{bladt@math.ku.dk}

\begin{document}
\maketitle


\begin{abstract}
In this paper, we consider statistical estimation of time--in\-ho\-mo\-ge\-neous aggregate Markov models. Unaggregated models, which corresponds to Markov chains, are commonly used in multi--state life insurance to model the biometric states of an insured. By aggregating microstates to each biometric state, we are able to model dependencies between transitions of the biometric states as well as the distribution of occupancy in these. This allows for non--Markovian modelling in general. Since only paths of the macrostates are observed, we develop an expectation--maximization (EM) algorithm to obtain maximum likelihood estimates of transition intensities on the micro level. Special attention is given to a semi-Markovian case, known as the reset property, which leads to simplified estimation procedures where EM algorithms for inhomogeneous phase--type distributions can be used as building blocks. We provide a numerical example of the latter in combination with piecewise constant transition rates in a three--state disability model with data simulated from a time--inhomogeneous semi--Markov model. Comparisons of our fits with more classic GLM-based fits as well as true and empirical distributions are provided to relate our model to existing models and their tools.      

\medskip


\medskip



\end{abstract}


\raggedbottom

\section{Introduction}
This paper considers statistical estimation of the finite state space, time--in\-ho\-mo\-ge\-neous aggregate Markov process introduced in the companion paper \cite{AhmadBladtFurrer}. The term aggregate refers to certain states of the process being pooled into macrostates. For example, the interpretation of macrostates could be biometric or behavioural states in a life insurance context, like active, disabled, or free--policy. The pooled states are referred to as microstates and are introduced to improve the sojourn time distributions of the macrostates and introduce dependencies between transitions. This allows for non--Markovian modelling in general. Since only paths of the macrostates are observed, this corresponds to an incomplete data problem with respect to the underlying microstates, and we employ an expectation--maximisation (EM) algorithm to obtain maximum likelihood estimates of transition intensities on the micro level.       

The aggregate Markov model of \cite{AhmadBladtFurrer} may be considered as the underlying process of a time--in\-ho\-mo\-ge\-neous  BMAP (Batch Markovian Arrival Processes, \cite{latouche1999introduction}), and contains as special cases time--ho\-mo\-ge\-neous phase--type renewal processes (see \cite{NeutsRenewal}), Markov modulated Poisson processes (see \cite{ryden94}) and of course (time--homogeneous) BMAPs. In the aggregate Markov model, the sojourn time distributions are inhomogeneous phase-type distributed (IPH, \cite{Albrecher-Bladt-2019}) and dependent, as shown in \cite[Proposition 4.1]{AhmadBladtFurrer}.\ Methods for fitting independent IPH distributions via the EM algorithm have been considered in \cite{Albrecher-Bladt-Yslas-2020} for commuting sub--intensity matrix functions and in \cite{IPH_Piecewise} for general IPHs. Methods for the estimation of homogeneous BMAPs can be found in \cite{breuer2002algorithm,breuer2003markov}. 

In the multi--state life insurance context, however, we both need the time--inhomogeneity and dependency between transitions to adequately capture age dependencies and duration effects. The nature of such models implies that sub--intensity matrices at different times may not commute, which is a crucial assumption in the independent IPH fitting of \cite{Albrecher-Bladt-Yslas-2020}.
We, therefore, extend the general approach of \cite{IPH_Piecewise} to include the dependencies. This will provide the main contribution of the paper. Special attention is paid to the semi--Markovian case considered in \cite[Subsection 4.2]{AhmadBladtFurrer}, known as the reset property, where sojourn time distributions are IPH and independent.\ Here, we show how algorithms of \cite{IPH_Piecewise} partly can be used as inputs to our algorithms.     

An important ingredient in our methods is the approximation of the models by piecewise constant transition rates. In general, transition probabilities of the Markov processes involved are solutions to ordinary differential equations of Kolmogorov type, the solution of which is denoted the product integral (see \cite{GillJohansen, Johansen}).\ Assuming piecewise constant rates, the solutions can be expressed explicitly in terms of products of matrix exponentials. Furthermore, maximum likelihood estimation greatly simplifies and can be expressed in terms of Multinomial and Poisson regressions based on sufficient statistics in the different time intervals.   

While we develop an EM algorithm for the general model, we only implement and apply it to data where the reset property is satisfied along with piecewise constant transition rates; this relates to the approach in \cite{IPH_Piecewise}.\ Here, we present a numerical example where macro data is simulated from a time--inhomogeneous semi-Markovian disability model commonly used in the context of disability insurance (see, e.g., \cite{hoem72, helwich, christiansen2012,BuchardtMollerSchmidt}). We compare our model fits with more classic GLM-based fits as well as true and empirical distributions to illustrate how the aggregate Markov model with the reset property is able to capture duration effects in these kinds of models.    

{\color{mogens}
The approach of using a GLM fit to a semi-Markovian model is a classical way that requires minimal implementation since GLM routines from standard software packages can be employed. The execution of the GLM is fast and can be applied to any semi--Markovian model with piecewise constant transition rates. On the other hand, aggregate Markov models are incomplete data models in which the GLM approach cannot be employed. Here, the EM algorithm is a standard method, which suffers from the usual issues related to any EM algorithm: slow convergence and the possibility of convergence to a local maximum or saddlepoint. Additionally, the current model suffers from overparameterisation and hence lack of identifiability. In spite of this, the potential applications of the aggregate Markov model with unobserved microstates are much broader and can, in principle, be used in any Markovian life insurance setting; see \cite{AhmadBladtFurrer} for further details.
}

The remainder of the paper is structured as follows. In Section \ref{sec:AggregatedMarkov}, we set up the model and notation. Section \ref{sec:complete_data} considers the estimation of completely observed aggregate Markov processes. Special attention is given to the piecewise constant case and the reset property, where links to Multinomial and Poisson regressions are provided. Then, in Section \ref{sec:em}, an EM algorithm for fitting aggregate Markov models from observing only the macro process is developed. Special attention is devoted to the piecewise constant case and the reset property. The proof of the EM algorithm is deferred to Appendix \ref{ap:proofs}. Finally, Section \ref{sec:num} contains a numerical example in a disability model.


\section{The aggregate Markov model}\label{sec:AggregatedMarkov}
We now present the aggregate Markov model introduced in the companion paper \cite{AhmadBladtFurrer} and some probabilistic properties of the model that are relevant to the present paper. {\color{mogens} Consider a jump process $X_1=\{X_1(t)\}_{t\geq 0}$ taking values on the finite set $\mathcal{J} = \{1,2,...,J\}$, $J\in \N$. We think of these as biometric or behavioural states governing the states of the insured in a life insurance context, for example, active, disabled, free--policy, or dead. We refer to these states as \textit{macrostates}.\ To each macrostate $i\in \mathcal{J}$, we assign a number $d_i \geq 1$ of so--called \textit{microstates}. If $d_i=1$, the macrostate coincides with the state $i$ of the Markov process $X_1$. If $d_i>1$, we think of a state $i$ as consisting of several underlying and unobserved microstates. Usually, the microstates are added with the sole purpose of improving the fit of the model, in particular, to capture duration effects, and the microstates are, therefore, not attributed to any physical interpretation. }

{\color{mogens}In order to keep track of both macro and microstates, we define a time-inhomogeneous Markov jump process  
$$\mat{X}=\{\mat{X}(t)\}_{t\geq 0}=\{(X_1(t), X_2(t))\}_{t\geq 0}$$
with state space
\begin{align*}
E=\{ \mat{\iay} = (i,\tilde{i}) : i\in \mathcal{J}, \, \tilde{i}\in \{1,2,...,d_i\}  \} .
\end{align*}
Here $X_2(t)$ is, in general, dependent on $X_1(t)$. For each $\mat{\iay} = (i,\tilde{i})$, $i$ denotes the current macrostate, which identifies a group of microstates, and $\tilde{i}$ denotes the actual microstate within this group. 
The total number of (micro)states is $|E|=\bar{d} = \sum_{i\in \J} d_i$. We assume that  $X_1(0) \equiv 1$, so that the initial distribution of $\mat{X}$ is given by
\begin{align*}
\vect{\pi} = (\vect{\pi}_1(0),\bm{0})
\end{align*}
for some initial distribution ($d_1$--dimensional row vector), $\mat{\pi}_1(0)$, the elements of which are the probabilities of $X_2(t)$ initiating in the different microstates, i.e. the distribution of $X_2(0)$. }

{\color{mogens}Denote by $\mat{M}(t)$ the intensity matrix function for $\mat{X}$.}
This can be written in the following block form: 
\begin{equation}
  \mat{M}(t) =
\begin{pmatrix}
\mat{M}_{11}(t) & \mat{M}_{12}(t) & \cdots & \mat{M}_{1J}(t) \\
 \mat{M}_{21}(t) & \mat{M}_{22}(t) & \cdots & \mat{M}_{2J}(t) \\
\vdots & \vdots & \ddots & \vdots \\
 \mat{M}_{J1}(t) & \mat{M}_{J2}(t) & \cdots & \mat{M}_{JJ}(t) \\
\end{pmatrix}\!, \label{eq:Lambda}
\end{equation}
where $\mat{M}_{ii}(t)$ are sub--intensity matrices of dimension $d_i \times d_i$ providing transition rates between the microstates {\color{mogens} within} macrostate $i$ at time $t$, and $\mat{M}_{ij}(t)$ are non-negative matrices of dimension $d_i \times d_j$ providing transition rates from microstates within macrostate $i$ to microstates within macrostate $j$ at time $t$. We denote an element of $\mat{M}(t)$ by $\mu_{\mat{\iay}\mat{\jay}}(t)$, $\mat{\iay}, \mat{\jay} \in E$.\ 

 {\color{mogens} The (defective) transition matrix for transitions between the microstates of macrostate $i$ only,} are given as the product integral (see \cite{Johansen,GillJohansen}) 
\begin{align}\label{eq:barP}
\bar{\mat{P}}_i(s,t) = \Prodi_s^t \!\left(\bm{I}+\mat{M}_{ii}(x)\md x\right), \quad i\in \mathcal{J}.
\end{align}
{\color{mogens} Throughout the paper, we let $\vect{e}_{\tilde{i}}$ denote a column vector (of appropriate dimension) having 1 in entry $\tilde{i}$ and zeros otherwise. Primes on vectors and matrices will denote matrix transposition.}

{\color{mogens}If we let $\vect{e}_{\tilde{i}}$ be of dimension $d_i$,} then $\vect{e}_{\tilde{i}}^{\prime} \bar{\mat{P}}_i(s,t)$ 
 contains the distribution of $\mat{X}(t)$ within macrostate $i$ on the event of staying in this macrostate {\color{mogens}throughout $[s,t]$, and conditional on $\mat{X}(s)=\mat{\iay}=(i,\tilde{i})$.}  The matrix $\mat{M}_{ij}(t){\rm d} t$, $j\neq i$, then contains the (infinitesimal) transition probabilities between microstates belonging to macrostates $i$ and $j$, respectively.

{\color{mogens} The paths of the macro process $X_1$ can be represented by its marked point process $(T_i,Y_i)_{i\in \mathbb{N}_0}$, where $T_i$ denotes the time of the $i$'th jump out of a macrostate, and where $Y_i$ denotes the state to which the jump occurs, i.e.\ $Y_i = X_1(T_i)$.
In particular, let  $\mathcal{S}_n = (T_i, Y_i)_{i\leq n}$ denote a sample path consisting of $n$ jumps and which terminates by a jump at time $T_n$ to state $Y_n$. If $\sai_n$ denotes a realisation of $\mathcal{S}_n$ with corresponding jump times $T_i=t_i$ and states $Y_i=y_i$}, then the $d_{y_n}$-dimensional row vector (see \cite[Lemma A.1]{AhmadBladtFurrer}) 
\begin{align}\label{eq:density_sn}
\mat{\alpha}(\sai_n)  
&= \mat{\pi}_1(0)
\prod_{\ell = 0}^{n-1}
\bar{\mat{P}}_{y_{\ell}}(t_{\ell},t_{\ell+1})\mat{M}_{y_{\ell}y_{\ell+1} }(t_{\ell +1})
\end{align} 
provides the {\color{mogens}rate of a jump to state $y_n$ at time $t_n$ on the event that previous jumps times and states are given by $t_i$ and $y_i$, $i\leq n-1$, respectively, and conditionally on initiating in macrostate 1.}

In particular, according to \cite[Remark 4.2]{AhmadBladtFurrer}, the sojourn times are inhomogeneous phase--type distributed (IPH, \cite{Albrecher-Bladt-2019}) and dependent on past jump times and transitions: 
\begin{align}\label{eq:sojourn_iph}
T_{n+1} - T_n \, | \, \mathcal{S}_n \sim \mathrm{IPH}\!\left(\frac{\mat{\alpha}(\mathcal{S}_n)}{\mat{\alpha}(\mathcal{S}_n)\vect{1}_{d_{y_n}} }, \, \mat{M}_{Y_nY_n}(T_n + \bigcdot)\right)\!.  
\end{align}
The corresponding exit rate function out of macrostate $i$ is then given as the column vector function
\begin{align}\label{eq:exit_rates}
\vect{m}_i(t) = -\mat{M}_{ii}(t)\vect{1}_{d_i} = \sum_{j\in \mathcal{J} \atop j\neq i} \mat{M}_{ij}(t)\vect{1}_{d_j},
\end{align}
{\color{mogens} where $\vect{1}_{d}$ denotes the $d$--dimensional column vector of ones. }

Throughout the paper, we pay special attention to the case where the \textit{reset property} introduced in \cite[(3.3)]{AhmadBladtFurrer} {\color{mogens} holds}. Here $\mat{M}_{ij}(t)$, $j \neq i$, {\color{mogens} is a rank one matrix} on the form
\begin{align}\label{eq:indep_cond}
\mat{M}_{ij}(t) = \vect{\beta}_{ij}(t) \vect{\pi}_j(t),
\end{align}
where $\vect{\beta}_{ij}(t)$ is a $d_i$--dimensional non-negative column vector function and $\vect{\pi}_j(t)$ is a $d_j$--dimensional non-negative row vector function with $\vect{\pi}_j(t) \vect{1}_{d_j} = 1$. {\color{mogens} Hence, coming from a microstate within macrostate $i$, the microstate within macrostate $j$ is picked indepently of all previous history. This can be seen as the process restarting or resetting itself at the time of such a transition. }

{\color{mogens} In case of resetting in all states}, 
\begin{align}\label{eq:exit_rates_indep}
\vect{m}_i(t) = \sum_{j\in \mathcal{J} \atop j\neq i}  \vect{\beta}_{ij}(t)\vect{\pi}_j(t)\vect{1}_{d_j} =  \sum_{j\in \mathcal{J} \atop j\neq i} \vect{\beta}_{ij}(t) .
\end{align}

{\color{mogens} 
\begin{remark}
In a life insurance context, the reset property \eqref{eq:indep_cond} 
means that leaving a macrostate $i$ by a jump out of a particular microstate does not influence which microstate is picked within the new macrostate $j$. Consequently, the total transition rate out of a macrostate $i$, \eqref{eq:exit_rates_indep}, is not influenced by the initiations of the target macrostates to which it jumps. \demoo
\end{remark}
}

Hence, according to \cite[Remark 4.5]{AhmadBladtFurrer}, the conditional sojourn time distributions \eqref{eq:sojourn_iph} become independent of past jump times and transitions: 
\begin{align}\label{eq:sojourn_iph_reset}
T_{n+1} - T_n \, | \, \mathcal{S}_n \sim \mathrm{IPH}\!\left(\vect{\pi}_{Y_n}(T_n), \, \mat{M}_{Y_nY_n}(T_n+\bigcdot)\right). 
\end{align}
These simplifications imply a specific time-inhomogeneous semi-Markovian structure to the macrostate process {\color{mogens} $X_1$}, cf.\ \cite[Subsection 4.2]{AhmadBladtFurrer}, which explains the focus on these type of models in our numerical example in Section \ref{sec:num}. 

In this paper, we develop methods for statistical fitting of the aggregate Markov model, namely estimation of the micro intensities $\mu_{\mat{\iay}\mat{\jay}}$ based on independent observations of the trajectories of the macrostate process {\color{mogens} $X_1$}.\ Since this leads to incomplete data with respect to the underlying macro-micro state process $\mat{X}$, we employ an expectation-maximization (EM) algorithm to obtain maximum likelihood estimations of the micro intensities. We develop a general EM algorithm and implement it in the case where the reset property \eqref{eq:indep_cond} is satisfied, along with piecewise constant transition rates.  

\subsection{Piecewise constant transition rates}\label{subsec:piecewise}
Following \cite[Section 2.1]{IPH_Piecewise},  suppose that the transition rates are piecewise constant on a grid $s_0 = 0 < s_1 < \cdots < s_{K-1} < \infty = s_K$ of $K$ time points, $K\in \mathbb{N}$, with values
\begin{align}\label{eq:rates_piecewise}
\begin{split}
\mat{M}(s) &= \mat{M}^k = \left\{ \mu_{\mat{\iay}\mat{\jay}}^k\right\}_{\mat{\iay},\mat{\jay}\in E}, \quad s\in (s_{k-1}, s_k], \quad k\in \{1,\ldots,K\}.
\end{split}
\end{align} 
Introducing $k(x)$ as the unique $k\in \{1,\ldots,K\}$ satisfying that $x\in (s_{k-1},s_k]$, we now have that the transition (sub--)probability matrix functions within macrostates \eqref{eq:barP} simplify to a product of matrix exponentials on the form 
\begin{align}\label{eq:barP_piecewise}
\bar{\mat{P}}_i(s,t) = {\rm e}^{\mat{M}_{ii}^{k(s)}\left(s_{k(s)}-s\right)}\!\left(\,\prod_{\ell = k(s)+1}^{k(t)-1}{\rm e}^{\mat{M}_{ii}^{\ell}(s_{\ell}-s_{\ell-1})} \right)\!  {\rm e}^{\mat{M}_{ii}^{k(t)}\left(t-s_{k(t)-1}\right)},
\end{align} 
with the convention that empty product integrals equal the identity matrix. The defective distribution \eqref{eq:density_sn} at time $t_n$ then also simplify, to  
\begin{align}\label{eq:density_sn_piecewise}
\mat{\alpha}(\sai_n) &= \mat{\pi}_1(0)
\prod_{\ell = 0}^{n-1}\bar{\mat{P}}_{y_{\ell}}(t_{\ell},t_{\ell+1})\mat{M}^{k(t_{\ell+1})}_{y_{\ell}y_{\ell + 1} },
\end{align} 
with the matrices $\bar{\mat{P}}_{y_\ell}(t_\ell, t_{\ell+1})$ being on the form \eqref{eq:barP_piecewise}.\ 
We give special attention to this case being satisfied along with the reset property when we develop our algorithms in this paper, as they will provide simplifications similar to those in \cite{IPH_Piecewise}.\  

Here it may be noted that if the resulting exit rates \eqref{eq:exit_rates} are different between two sub--intervals, the density of the conditional sojourn time distributions \eqref{eq:sojourn_iph} become discontinuous at the corresponding grid point between the two sub--intervals. This follows by similar arguments as those made in \cite[Subsection 2.1]{IPH_Piecewise}.\    


\section{The case of complete micro data}\label{sec:complete_data}
We now consider the complete data case where trajectories of the underlying macro-micro state process $\mat{X}$ are fully observed, which corresponds to methods known from inference of time-inhomogeneous Markov jump processes on finite state spaces; we refer to \cite{Andersen} for a detailed exposition on this. The approach and notation of this section largely follow that of \cite[Section 3.1--3.2]{IPH_Piecewise}. 

{\color{mogens}
  In this paper, the computational approach we take, and implement, uses an approximation of transition rates by piecewise constant functions and assumes the reset property to avoid a full--trajectory history. First, we set up the general framework, where the basic formulas are presented.  }

\subsection{{\color{mogens}Notation and general considerations}}\label{subsec:comple_data_general}
Suppose that we observe $N\in \N$ i.i.d.\ realizations of the Markov jump process $\mat{X}$ of macro--micro states on some time interval $[0,T]$, where $T>0$ is a given and fixed time horizon; represent the (fictive) data by $\mat{X} = (\mat{X}^{(1)},\ldots,\mat{X}^{(N)})$.\ Denote with $\mat{N} = (N^{(1)},\ldots, N^{(N)})$ the corresponding data of the multivariate counting process associated to $\mat{X}$, that is, $N^{(n)}$, $n=1,\ldots,N$, have components
\begin{align*}
N^{(n)}_{\mat{\iay}\mat{\jay}}(t)=\# \tuborg{s\in (0,t] : \mat{X}^{(n)}(s-)=\mat{\iay}, \ \mat{X}^{(n)}(s)=\mat{\jay}}.
\end{align*}
Parametrising the transition rates on the micro level with a parameter vector $\bm{\theta}\in \bm{\Theta}$, where $\bm{\Theta}$ is some finite--dimensional, suitably regular parameter space with non--empty interior, such that  
\begin{align*}
\mat{M}(s) &= \mat{M}(s; \mat{\theta}),
\end{align*}
we have that the likelihood function for the joint parameter $(\vect{\pi}_1, \mat{\theta})$ is given by
\begin{align}\label{eq:likelihood_general}
\begin{split}
\mathcal{L}^{\mat{X}}(\vect{\pi}_1, \mat{\theta}) 
&= 
\mathcal{L}^{\mat{X}}_0(\vect{\pi}_1)
\prod_{\mat{\iay},\mat{\jay}\in E \atop \mat{\jay}\neq \mat{\iay} } 
\mathcal{L}^{\mat{X}}_{\mat{\iay}\mat{\jay}}(\mat{\theta}), 
\\[0.2 cm]
\mathcal{L}^{\mat{X}}_0(\vect{\pi}_1) 
&=  
\prod_{r= 1}^{d_1}\pi_{(1,r)}(0)^{B_{(1,r)}(0) }, \\[0.2 cm]
\mathcal{L}^{\mat{X}}_{\mat{\iay}\mat{\jay}}(\mat{\theta}) 
&= 
\exp\bigg(\int_{(0,T]} \log\!\left(\mu_{\mat{\iay}\mat{\jay} }(s;\bm{\theta})\right)\!\!\md N_{\mat{\iay}\mat{\jay}}(s)  - \int_0^T I_{\mat{\iay}}(s) \mu_{\mat{\iay}\mat{\jay}}(s;\bm{\theta})\!\md s    \bigg),
\end{split}
\end{align}
where {\color{mogens}$B_{(1,r)}(0)$ is the number of processes initiating in state $(1,r)$, and}, for $\mat{\iay},\mat{\jay}\in E$, $\mat{\jay}\neq \mat{\iay}$, and $s\in [0,T]$, 
\begin{align}\label{eq:start_obs_complete}
I_{\mat{\iay}}(s) = \sum_{n=1}^N \mathds{1}_{(\mat{X}^{(n)}(s) = \mat{\iay})  } \qquad \text{and} \qquad 
N_{\mat{\iay}\mat{\jay}}(s) = \sum_{n=1}^N N^{(n)}_{\mat{\iay}\mat{\jay}}(s) .
\end{align}

Then $I_{\mat{\iay}}(s)$ gives the number of observations in  state $\mat{\iay}$ at time $s$, while $N_{\mat{\iay}\mat{\jay}}(s)$ gives the total number of jumps observed from state $\mat{\iay}$ to $\mat{\jay}$ on $[0,s]$.

The corresponding log-likelihood $L^{\mat{X}}(\vect{\pi}_1, \mat{\theta}) = \log \mathcal{L}^{\mat{X}}(\vect{\pi}_1, \mat{\theta})$ is  given by\medskip
\begin{align}\label{eq:log-likelihood_general}
L^{\mat{X}}(\vect{\pi}_1, \mat{\theta}) &= 
L^{\mat{X}}_0(\vect{\pi}_1)+\sum_{\mat{\iay},\mat{\jay}\in E \atop \mat{\jay}\neq \mat{\iay} } L^{\mat{X}}_{\mat{\iay}\mat{\jay}}(\mat{\theta}), \\[0.2 cm]\nonumber
L^{\mat{X}}_0(\vect{\pi}_1) &=  \sum_{r= 1}^{d_1}B_{(1,r)}(0) \log(\pi_{(1,r)}(0)), \\[0.2 cm]\label{eq:log-likelihood_general_statewise}
L^{\mat{X}}_{\mat{\iay}\mat{\jay}}(\mat{\theta}) &= \int_{(0,T]} \log\!\left(\mu_{\mat{\iay}\mat{\jay} }(s;\bm{\theta})\right)\!\!\md N_{\mat{\iay}\mat{\jay}}(s)  -\int_0^T I_{\mat{\iay}}(s) \mu_{\mat{\iay}\mat{\jay}}(s;\bm{\theta})\!\md s,
\end{align}
and the MLE of $(\vect{\pi}_1, \mat{\theta})$ is then found by maximizing the log-likelihood: 
\begin{align*}
( \hat{\vect{\pi}}_1, \hat{\mat{\theta}}) = \underset{(\vect{\pi}_1,\, \mat{\theta})}{\mathrm{arg\, max}}  \ L^{\mat{X}}(\vect{\pi}_1, \mat{\theta}).
\end{align*}
The product structure of the likelihood  \eqref{eq:likelihood_general} in $\vect{\pi}_1$ and $\mat{\theta}$ gives that we can estimate these separately via their respective likelihoods $\mathcal{L}^{\mat{X}}_0$ and $\mathcal{L}_{\mat{\iay}\mat{\jay}}^{\mat{X}}$, $\mat{\iay},\mat{\jay}\in E$, $\mat{\jay}\neq \mat{\iay}$.\ For $\vect{\pi}_1$, one realizes that the likelihood $\mathcal{L}^{\mat{X}}_0$ is proportional to the likelihood obtained from viewing $(B_{(1,1)}(0),\ldots,B_{(1,d_1)}(0))$ as an observation from the $\mathrm{Multinomial}(N,\vect{\pi}_1(0))$--distribution, where $N$ is considered fixed. Hence, the MLE of $\vect{\pi}_1$ is explicitly given by
\begin{align}\label{eq:mle_pi}
\hat{\pi}_{(1,r)}(0) &= \frac{B_{(1,r)}(0)}{N}.
\end{align}
The MLE of $\mat{\theta}$ is then given by
\begin{align*}
\hat{\mat{\theta}} &= \underset{\mat{\theta}}{\mathrm{arg\, max}}  \sum_{\mat{\iay},\mat{\jay}\in E \atop \mat{\jay}\neq \mat{\iay} } L^{\mat{X}}_{\mat{\iay}\mat{\jay}}(\mat{\theta}),
\end{align*}  
which, in general, requires numerical optimisation methods. Similar discussions are seen in, e.g., \cite[Section 3.1]{IPH_Piecewise}.

\subsection{Reset property}\label{subsec:comple_data_reset}
We now assume the reset property \eqref{eq:indep_cond} is satisfied. The setup remains that of Subsection \ref{subsec:comple_data_general}, except that we now, due to the nature of the exit rates $\vect{\beta}_{ij}$ and initial distributions $\vect{\pi}_j$ playing distinct roles, extend the parameter space to $\mat{\Theta}\times \mat{H}$, such that, for $(\mat{\theta}, \vect{\eta})\in \mat{\Theta}\times \mat{H}$,
\begin{align}
\begin{split}\label{eq:reset_parameter_general}
\mat{M}_{ij}(s; \mat{\theta},\vect{\eta}) &= \vect{\beta}_{ij}(s;\mat{\theta})\vect{\pi}_{j}(s;\vect{\eta}), \qquad j\neq i.
\end{split}
\end{align}
This parameterisation allows for separate estimations of exit rates and initial distributions within the reset property. Note that we implicitly also set $\vect{\pi}_1(0) = \vect{\pi}_1(0; \vect{\eta})$, so that we allow for the possibility of $\vect{\pi}_1(0)$ to be regressed against the other initial distributions at the different time points.  

Having this setup, we see by splitting the 
likelihood contributions for the different transitions, $\mathcal{L}^{\mat{X}}_{\mat{\iay}\mat{\jay}}$, between those within macrostates and those between macrostates that the likelihood \eqref{eq:likelihood_general} now simplifies to, using that $\vect{\pi}_j(s; \vect{\eta})\vect{1}_{d_j} = 1$ for all $j$,
\begin{align*}
\mathcal{L}^{\mat{X}}(\mat{\theta}, \vect{\eta}) &= 
\prod_{\mat{\iay}\in E}\mathcal{L}^{\mat{X}}_{\mat{\iay}}(\vect{\eta})\prod_{\check{i} = 1\atop \check{i}\neq \tilde{i}}^{d_i} \mathcal{L}^{\mat{X}}_{\mat{\iay}(i,\check{i})}(\mat{\theta})\prod_{j\in \mathcal{J}\atop j\neq i}\mathcal{L}^{\mat{X}}_{\mat{\iay}j}(\mat{\theta}),
\end{align*}
with $\mathcal{L}^{\mat{X}}_{\mat{\iay}(i,\check{i})}$ as in \eqref{eq:likelihood_general}, and 
\begin{align}\label{eq:likelihood_reset_init_jump}
\mathcal{L}^{\mat{X}}_{\mat{\iay}}(\vect{\eta}) &= \exp\bigg(\int_{[0,T]} \log\left(\pi_{\mat{\iay}}(s ; \vect{\eta})\right)\!\!\md N_{\mat{\iay}}(s)\bigg),\\[0.2 cm]\label{eq:likelihood_reset_init_jump2}
\mathcal{L}^{\mat{X}}_{\mat{\iay}j}(\mat{\theta}) &= \exp\bigg(\int_{(0,T]} \log\left(\beta_{\mat{\iay}j}(s ; \mat{\theta})\right)\!\!\md N_{\mat{\iay}j}(s)-\int_0^T I_{\mat{\iay}}(s) \beta_{\mat{\iay}j}(s ; \mat{\theta})\md s\bigg),
\end{align}
where we define $N_{\mat{\iay}j}$ and $N_{\mat{\iay}}$  as the aggregated processes
\begin{align}\label{eq:agg_N_reset_general}
N_{\mat{\iay}j}(s) = \sum_{\tilde{j} = 1}^{d_j}N_{\mat{\iay}\mat{\jay}}(s)
\quad   \text{and} \quad 
N_{\mat{\iay}}(s) &= \begin{cases}
\sum_{\mat{\jay}\in E\atop j\neq i}N_{\mat{\jay}\mat{\iay}}(s)  \quad
&\text{for} \ s>0, \\[0.2 cm]
B_{\mat{\iay}}(0) &\text{for}\ s=0\ \text{and}\ i=1,\\[0.2 cm]
0 &\text{Otherwise}.
\end{cases}
\end{align}
Note that we rather untraditionally, but for notational convenience, couple the number of initiations at time $0$ with the counting process counting the number of jumps into a macrostate in our definition of $N_{\mat{\iay}}$.\ This is related to the aforementioned possibility of regressing the initial distribution $\vect{\pi}_1(0;\vect{\eta})$ at time $0$ with the other initial distributions $\vect{\pi}_j(\cdot ; \vect{\eta})$, where this definition allows us to unify computations.   

The corresponding log-likelihood \eqref{eq:log-likelihood_general} takes the form
\begin{align}\label{eq:log-likelihood_reset}
L^{\mat{X}}(\mat{\theta}, \vect{\eta}) &= 
\sum_{\mat{\iay}\in E}\Bigg(L^{\mat{X}}_{\mat{\iay}}(\vect{\eta}) + \sum_{\check{i} = 1\atop \check{i}\neq \tilde{i}}^{d_i} L^{\mat{X}}_{\mat{\iay}(i,\check{i})}(\mat{\theta}) + \sum_{j\in \mathcal{J}\atop j\neq i}L^{\mat{X}}_{\mat{\iay}j}(\mat{\theta})   \Bigg),
\end{align}
with $L^{\mat{X}}_{\mat{\iay}(i,\check{i})}$ as in \eqref{eq:log-likelihood_general_statewise}, and 
\begin{align}
\begin{split}\label{eq:log-likelihood_reset_init_jump}
L^{\mat{X}}_{\mat{\iay}}(\vect{\eta}) &=  \int_{[0,T]} \log\left(\pi_{\mat{\iay}}(s ; \vect{\eta})\right)\!\!\md N_{\mat{\iay}}(s),\\[0.2 cm]
L^{\mat{X}}_{\mat{\iay}j}(\mat{\theta}) &=  \int_{(0,T]} \log\left(\beta_{\mat{\iay}j}(s ; \mat{\theta})\right)\!\!\md N_{\mat{\iay}j}(s)-\int_0^T I_{\mat{\iay}}(s) \beta_{\mat{\iay}j}(s ; \mat{\theta})\md s.
\end{split}
\end{align}
The MLE of $(\mat{\theta},\vect{\eta})$ is then found by maximizing the relevant log-likelihood contributions:\ 
\begin{align*}
\hat{\mat{\theta}} &= 
\underset{\mat{\theta}}{\mathrm{arg\, max}}  
\sum_{\mat{\iay}\in E}\Bigg(\sum_{\check{i} = 1\atop \check{i}\neq \tilde{i}}^{d_i} L^{\mat{X}}_{\mat{\iay}(i,\check{i})}(\mat{\theta}) + \sum_{j\in \mathcal{J}\atop j\neq i}L^{\mat{X}}_{\mat{\iay}j}(\mat{\theta})\Bigg), \\[0.3 cm]
\hat{\vect{\eta}} &= \underset{\vect{\eta}}{\mathrm{arg\, max}}\, \sum_{\mat{\iay}\in E} L^{\mat{X}}_{\mat{\iay}}(\vect{\eta}),
\end{align*}
which also, here, in general, requires numerical optimisation methods. 
\subsection{Piecewise constant transition rates}\label{subsec:comple_data_piecewise_general}
Consider again the general case of Subsection \ref{subsec:comple_data_general}, and assume now that the transition intensity matrix function $\mat{M}(\cdot ; \mat{\theta})$ is piecewise constant on the form \eqref{eq:rates_piecewise}.\ Then the likelihood contributions for the transitions between states, $\mathcal{L}^{\mat{X}}_{\mat{\iay}\mat{\jay}}$,  simplify to: 
\begin{align}\label{eq:likelihood_piecewise_statewise}
\mathcal{L}^{\mat{X}}_{\mat{\iay}\mat{\jay}}(\mat{\theta}) &= \prod_{k=1}^{K} \left(\mu_{\mat{\iay}\mat{\jay}}^k( \mat{\theta})\right)^{O_{\mat{\iay}\mat{\jay}}(k)}\exp\!\left(-\mu_{\mat{\iay}\mat{\jay}}^k( \mat{\theta})E_{\mat{\iay}}(k)\right)\!,
\end{align} 
where $O_{\mat{\iay}\mat{\jay}}(k)$ is the total number of \textit{occurrences} of transitions from state $\mat{\iay}$ to $\mat{\jay}$ in the time interval $(s_{k-1},s_{k}]$, and $E_{\mat{\iay}}(k)$ is the total time spent in state $\mat{\iay}$ in the time interval $(s_{k-1},s_{k}]$, the so-called \textit{exposure}, given by
\begin{align}
\begin{split}\label{eq:OE}
O_{\mat{\iay}\mat{\jay}}(k) =  \int_{(s_{k-1},s_k]} \md N_{\mat{\iay}\mat{\jay}}(t) \qquad \text{and} \qquad
E_{\mat{\iay}}(k) =  \int_{s_{k-1}}^{s_{k}} I_{\mat{\iay}}(t)  \md t.
\end{split}
\end{align}
The corresponding log--likelihood contributions take the form 
\begin{align}\label{eq:log-likelihood_piecewise_statewise}
L^{\mat{X}}_{\mat{\iay}\mat{\jay}}(\mat{\theta}) &= 
\sum_{k=1}^{K} 
\left(O_{\mat{\iay}\mat{\jay}}(k)\log\!\big(\mu_{\mat{\iay}\mat{\jay}}^k( \mat{\theta})\big)-\mu_{\mat{\iay}\mat{\jay}}^k( \mat{\theta})E_{\mat{\iay}}(k)\right)\!.
\end{align}
Thus, in the case of piecewise constant transition rates, the occurrences and exposures in the different time intervals, along with the number of initiations in the different microstates of macrostate 1, 
\begin{align*}
\left\{
\left(B_{(1,r)}(0), \,
	  O_{\mat{\iay}\mat{\jay}}(k),\, 
	  E_{\mat{\iay}}(k) 
\right)
\right\}_{
k\in\{1,\ldots,K\},\, r\in\{1,\ldots,d_1\},\, \mat{\iay},\mat{\jay}\in E, \, \mat{\jay}\neq \mat{\iay}
},
\end{align*}
are sufficient statistics. One even notes that the resulting likelihood, \eqref{eq:likelihood_general} combined with \eqref{eq:likelihood_piecewise_statewise}, is proportional to the likelihood obtained from independent observations 
\begin{align}\label{eq:stikproever_complete_general}
\begin{split}
&\left(B_{(1,1)}(0), \ldots, B_{(1,d_1)}(0)\right)\!, \\[0.2 cm]
&\left(O_{\mat{\iay}\mat{\jay}}(k),\quad  k\in\{1,\ldots,K\}, \ \ \mat{\iay},\mat{\jay}\in E,\ \mat{\jay}\neq \mat{\iay}\right)\!, 
\end{split}
\end{align}
where 
\begin{align}\label{eq:regr_complete_general}
\begin{split}
\left(B_{(1,1)}(0), \ldots, B_{(1,d_1)}(0)\right) \quad &  \text{is} \quad \mathrm{Multinomial}(N, \bm{\pi}_1(0)) - \text{distributed},\\[0.2 cm]
O_{\mat{\iay}\mat{\jay}}(k) \quad  & \text{is} \quad \mathrm{Poisson}\!\left(E_{\mat{\iay}}(k)\mu_{\mat{\iay}\mat{\jay}}^k(\mat{\theta})\right) - \text{distributed}, 
\end{split}
\end{align}
with $N$ and $E_{\mat{\iay}}(k)$ considered fixed. Hence, the MLE of $\mat{\pi}_1(0)$ remains explicitly given by \eqref{eq:mle_pi}, while the MLE of $\mat{\theta}$ can be obtained from Poisson regressions of the occurrences against the times on the grid, which can be carried using standard software packages. For example, if the intensities $\mu_{\mat{\iay}\mat{\jay}}^k(\mat{\theta})$ are exponential functions of $\mat{\theta}$, a Poisson regression with log--link function and log--exposure as offset can be carried out, corresponding to fitting the model  
\begin{align}\label{eq:pois-regr}
\log (\mu_{\mat{\iay}\mat{\jay}}(s; \mat{\theta})) 
= 
\sum_{r = {\color{mogens}0}}^q \theta_{\mat{\iay}\mat{\jay}}^{(r)}g^{(r)}(s),
\end{align}
for suitably regular known functions $g^{(r)}$, with a common choice being $g^{(r)}(s) = s^r${\color{mogens}, and where $\mat{\theta}=\left( \theta_{\mat{\iay}\mat{\jay}}^{(r)} \right)_{\mat{\iay}\mat{\jay}\in E,\, \mat{\jay}\neq \mat{\iay},\, r=0,...,q}$}.\ The predictions at $s_k$ then constitutes the MLEs of the intensities. 
 
In the special case where each of the parameters in $\mat{\theta}$ are the transition rates in the different time intervals, that is, $\mat{\theta} = \left(\theta_{\mat{\iay}\mat{\jay}}^k\right)_{
\mat{\iay},\mat{\jay}\in E, \, \mat{\jay}\neq \mat{\iay},\, k=1,\ldots,K 
}$ such that
\begin{align*}
\mu_{\mat{\iay}\mat{\jay}}^k(\mat{\theta}) &= \theta_{\mat{\iay}\mat{\jay}}^k,
\end{align*} 
the MLE of $\mat{\theta}$ simplify to so-called occurrence--exposure rates (cf.\ also \cite{AsmussenEM, IPH_Piecewise}):
\begin{align*} 
\hat{\theta}_{\mat{\iay}\mat{\jay}}^k = \frac{O_{\mat{\iay}\mat{\jay}}(k)}{E_{\mat{\iay}}(k)}.
\end{align*} 
This can be seen as a direct ``non--parametric'' approach to estimate the micro intensities in the different time intervals, which then is a special case of the general parametric approach. The assumption of piecewise constant transition rates is often seen as an approximation to continuous versions obtained when the number of grid points tends to infinity. The observations of this subsection largely follow the remarks made in \cite[Section 5]{Aalen2008}. 

\subsection{Reset property with piecewise constant transition rates}\label{subsec:comple_data_piecewise_reset}
We now assume the reset property \eqref{eq:indep_cond} in combination with piecewise constant transition rates on the form \eqref{eq:rates_piecewise}, so that for $j\neq i$, $k\in \{1,\ldots,K\}$, and $s\in (s_{k-1},s_k]$,
\begin{align*}
\vect{\beta}_{ij}(s; \mat{\theta}) &= \vect{\beta}_{ij}^k(\mat{\theta}), \\[0.2 cm]
\vect{\pi}_{j}(s; \vect{\eta}) &= \vect{\pi}_{j}^k(\vect{\eta}),
\end{align*}
with $\vect{\pi}_{1}(0;\vect{\eta}) = \vect{\pi}_{1}^0(\vect{\eta})$.\ The transition rates between macrostates are then on the form
\begin{align}\label{eq:reset_piecewise}
\mat{M}^k_{ij}(\mat{\theta}, \vect{\eta}) &= \vect{\beta}_{ij}^k(\mat{\theta})\vect{\pi}_{j}^k(\vect{\eta}).
\end{align}
In this case, the likelihood contributions for transitions between macrostates \eqref{eq:likelihood_reset_init_jump}--\eqref{eq:likelihood_reset_init_jump2} simplify to  
\begin{align}\label{eq:likelihood_reset_piecewise_statewise}
\begin{split}
\mathcal{L}^{\mat{X}}_{\mat{\iay}}(\vect{\eta}) &= \prod_{k=0}^K \pi_{\vect{\iay}}^k(\vect{\eta})^{B_{\mat{\iay}}(k)}, \\[0.2 cm]
\mathcal{L}^{\mat{X}}_{\mat{\iay}j}(\mat{\theta}) &= \prod_{k=1}^K \beta_{\mat{\iay}j}^k(\mat{\theta})^{O_{\mat{\iay}j}(k)}\exp\!\left(-\beta_{\mat{\iay}j}^k(\mat{\theta})E_{\mat{\iay}}(k)\right)\!, \\[0.2 cm]
\end{split}
\end{align}
where, for $k\in \{1,\ldots,K\}$, $B_{\mat{\iay}}(k)$ is the total number of initiations in microstate $\tilde{i}$ in the time interval $(s_{k-1},s_k]$ resulting from jumps into macrostate $i$, and $O_{\mat{\iay}j}(k)$ is the
total number of transitions in time interval $(s_{k-1},s_k]$ from macrostate $i$ to $j$ happening from microstate $\tilde{i}$:
\begin{align}\label{eq:Occurrences_reset}
B_{\mat{\iay}}(k) = \sum_{\mat{\jay}\in E\atop j\neq i} O_{\mat{\jay}\mat{\iay}}(k)\qquad \text{and} \qquad O_{\mat{\iay}j}(k) = \sum_{\tilde{j} = 1}^{d_j} O_{\mat{\iay}\mat{\jay}}(k) .
\end{align}

The corresponding log--likelihood contributions simplify to 
\begin{align}\label{eq:log-likelihood_reset_piecewise_statewise}
L^{\mat{X}}_{\mat{\iay}}(\vect{\eta}) &= \sum_{k=0}^K B_{\mat{\iay}}(k)\log\!\big(\pi_{\vect{\iay}}^k(\vect{\eta})\big), \\[0.2 cm]\label{eq:log-likelihood_reset_piecewise_statewise2}
L^{\mat{X}}_{\mat{\iay}j}(\mat{\theta}) &= \sum_{k=1}^K \left(O_{\mat{\iay}j}(k)\log\!\big(\beta_{\mat{\iay}j}^k( \mat{\theta})\big)-\beta_{\mat{\iay}j}^k( \mat{\theta})E_{\mat{\iay}}(k)\right)\!. 
\end{align}
Consequently, in the case where the reset property is satisfied in combination with piecewise constant transition rates, \eqref{eq:rates_piecewise} and \eqref{eq:reset_piecewise}, the sufficient statistics regarding the occurrences between macrostates reduce to those of \eqref{eq:Occurrences_reset}.\ In fact, by inserting the simplified likelihood contributions \eqref{eq:likelihood_reset_piecewise_statewise}   into the general piecewise constant case \eqref{eq:likelihood_piecewise_statewise}, which again are inserted into the general likelihood \eqref{eq:likelihood_general}, we realise that it now is proportional to the likelihood obtained from independent observations 
\begin{align*}
&\left(B_{\mat{\iay}}(k),\qquad\qquad\qquad\ \  k=0,\ldots,K, \ \ \mat{\iay}\in E\right)\!, \\[0.2 cm]
&\left(\big(O_{\mat{\iay}(i,\check{i})}(k), O_{\mat{\iay}j}(k)\big),   \quad  k=1,\ldots,K, \ \ \mat{\iay}\in E, \ \check{i}\in \{1,\ldots,d_i\}, \ \check{i}\neq \tilde{i},\  j\in \mathcal{J},\  j\neq i\right)\!, 
\end{align*}
where  
\begin{align*}
\left(B_{(i,1)}(k),\ldots,B_{(i,d_i)}(k)\right) \quad  & \text{is} \quad \mathrm{Multinomial}\!\left(B_i(k), \bm{\pi}^k_i(\vect{\eta})\right) - \text{distributed},\\[0.2 cm]
O_{\mat{\iay}(i,\check{i})}(k) \quad  & \text{is} \quad \mathrm{Poisson}\!\left(E_{\mat{\iay}}(k)\mu_{\mat{\iay}(i,\check{i})}^k(\mat{\theta})\right) - \text{distributed},
\\[0.2 cm]
O_{\mat{\iay}j}(k)\quad  & \text{is} \quad \mathrm{Poisson}\!\left(E_{\mat{\iay}}(k)\beta_{\mat{\iay}j}^k(\mat{\theta})\right) - \text{distributed},
\end{align*}
with $B_i(k) = \sum_{\tilde{i}=1}^{d_i} B_{\mat{\iay}}(k)$, $k\in \{1,\ldots,K\}$, being the total number of jumps to macrostate $i$ observed in $(s_{k-1},s_k]$;\ for $k=0$ and $i = 1$, we have $B_1(0) = N$, cf.\ also \eqref{eq:regr_complete_general}. Hence, the MLE of $\mat{\theta}$ is obtained from similar kinds of Poisson regressions as those in Subsection \ref{subsec:comple_data_piecewise_general}, but the MLE of $\vect{\eta}$ can now be obtained from Multinomial regressions of the number of initiations against the times on the grid, which also can be carried using standard software packages. 

For example, if for a fixed macrostate $i\in \mathcal{J}$, the probabilities $\pi_{\mat{\iay}}^k(\vect{\eta})$ are exponential functions of $\vect{\eta}$ (relative to the probability $\pi_{(i,d_i)}^k(\vect{\eta})$ in the last microstate $d_i$, say), then a Multinomial logistic regression can be carried out, corresponding to fitting the model 
\begin{align}\label{eq:mult-regr}
\pi_{\mat{\iay}}(s; \vect{\eta}) = \frac{
\exp\!\left(\mathds{1}_{(\tilde{i}\neq d_i)}\sum_{r={\color{mogens} 0} }^q \eta_{\mat{\iay}}^{(r)}g^{(r)}(s)    \right)
}{
1+\sum_{\check{i}=1}^{d_i-1}
\exp\!\left(\sum_{r={\color{mogens} 0} }^q \eta_{(i,\check{i})}^{(r)}g^{(r)}(s)    \right)
},
\end{align}
where the functions $g^{(r)}$ are as in \eqref{eq:pois-regr}.\ The predictions at $s_k$ then constitute the MLEs of the initial distributions.  A similar type of Multinomial logistic regressions for initial distributions of (inhomogeneous) phase--type distributions are performed in \cite{bladtyslas_experts, Albrecher-Bladt-Muller}, although in the context of covariate information.


\section{EM algorithm for the aggregate Markov model}\label{sec:em}

In this section, we give the main contributions of the paper, namely the maximum likelihood
estimation of micro intensities using the expectation-maximization (EM) algorithm. {\color{mogens} While we only present the algorithm, and apply to data, in the case where the reset property holds, we start by introducing some general results that will both set up the notation and clarify the importance of the reset property. This will extend 
the results of \cite{IPH_Piecewise}.
}         
%
%

\subsection{{\color{mogens}Notation and general considerations}}\label{subsec:EM_general}

The macro data we observe are $N$ i.i.d.\ realisations of the macrostate process {\color{mogens} $X_1$} on the generic time interval $[0,T]$.\ It shall be useful to represent the data via the associated marked point process $(T_i,Y_i)_{i\in \N_0}$ to keep track of jump times and transitions. This is also the approach of, e.g.\ \cite{AsmussenEM}. 

Denote with $\mathcal{S}^{(n)} = \left(T^{(n)}, Y^{(n)}\right) = (T_i^{(n)}, Y_i^{(n)})_{i\leq M^{(n)}}$ the $n$'th observation, $n=1,\ldots,N$, of jump times and transitions of the macrostate process {\color{mogens} $X_1$}, where $M^{(n)}$ is the number of transitions observed, so that $T = \max_{n=1,\ldots,N} T^{(n)}_{M^{(n)}}$.\ Represent all observed data by the vector  
\begin{align*}
\vect{\mathcal{S}} = \left(\mathcal{S}^{(1)},\ldots,\mathcal{S}^{(N)}\right)\!.
\end{align*}
Let $\mathbb{E}_{(\vect{\pi}_1, \mat{\theta})}$ denote the expectation under which the Markov jump process $\mat{X}$ of macro-micro states has transition intensity matrix function $\mat{M}(\cdot ; \mat{\theta})$ and initial distribution $(\vect{\pi}_1(0), \vect{0})$.\ The EM algorithm for estimation of the micro--level parameter $(\vect{\pi}_1, \mat{\theta})$ then consists of initialising with some value $(\vect{\pi}^{(0)}_1, \mat{\theta}^{(0)}) \in [0,1]^{d_1}\times \mat{\Theta}$, and then iteratively compute the conditional expected log-likelihood given macro data under some current parameter $(\vect{\pi}^{(m)}_1,\, \mat{\theta}^{(m)})$, the so-called E--step,
\begin{align}\label{eq:cond_log-likelihood_general_def}
(\vect{\pi}_1, \mat{\theta}) \mapsto \bar{L}^{(m)}(\vect{\pi}_1,\, \mat{\theta}) = \mathbb{E}_{(\vect{\pi}^{(m)}_1,\, \mat{\theta}^{(m)})}\!\!\left[\left.  L^{\mat{X}}(\vect{\pi}_1, \mat{\theta}) \, \right|  \vect{\mathcal{S}} \right]\!, \quad m\in \mathbb{N}_0,       
\end{align}
and then maximize this to update the parameter to $(\vect{\pi}^{(m+1)}_1, \mat{\theta}^{(m+1)})$, the so-called M--step. For notational convenience, we write, under some parameter $(\vect{\pi}^{(m)}_1, \mat{\theta}^{(m)})$,
\begin{align}\label{eq:barP_theta}
\bar{\mat{P}}^{(m)}_i(s,t) &= \Prodi_s^t \!\left(\bm{I}+\mat{M}_{ii}\big(x; \mat{\theta}^{(m)}\big)\!\md x\right)\!,\quad i\in \mathcal{J},
\end{align}
for the transition (sub--)probability matrix functions within macrostates, and 
\begin{align}\label{eq:density_sn_m}
\mat{\alpha}^{(m)}(\sai_n) &= \mat{\pi}^{(m)}_1(0)
\prod_{\ell = 0}^{n-1}
\bar{\mat{P}}^{(m)}_{y_{\ell}}(t_{\ell},t_{\ell+1})\mat{M}_{y_{\ell}y_{\ell+1} }\big(t_{\ell +1}; \mat{\theta}^{(m)}\big),
\end{align} 
for the corresponding defective distribution at time $t_n$.\ Also, we denote with $\vect{1}_n^{\prime}$ the row vector of ones with the same dimension as $\vect{\alpha}^{(m)}(\mathcal{S}^{(n)})$.  

To obtain the conditional expected log-likelihood given macro data, we need some conditional expected statistics. For $r\in \{1,\ldots,d_1\}$, and $\mat{\iay},\mat{\jay}\in E$, $\mat{\jay}\neq \mat{\iay}$, define 
\begin{align}\label{eq:cond_statistics_general}
\bar{B}^{(m)}_{(1,r)}(0) 
&= 
\mathbb{E}_{(\vect{\pi}^{(m)}_1,\, \mat{\theta}^{(m)})}\!\!\left[\left. B_{(1,r)}(0) \, \right|  \vect{\mathcal{S}} \right]\!, \\[0.2 cm]\label{eq:cond_statistics_general_I}
\bar{I}^{(m)}_{\mat{\iay}}(s) &= \mathbb{E}_{(\vect{\pi}^{(m)}_1,\, \mat{\theta}^{(m)})}\!\left[\left. I_{\mat{\iay}}(s) \, \right|  \vect{\mathcal{S}} \right]\!, \\[0.2 cm]\label{eq:cond_statistics_general_N}
\bar{N}^{(m)}_{\mat{\iay}\mat{\jay}}(s) &= \mathbb{E}_{(\vect{\pi}^{(m)}_1,\, \mat{\theta}^{(m)})}\!\left[\left. N_{\mat{\iay}\mat{\jay}}(s) \, \right|  \vect{\mathcal{S}} \right]\!. 
\end{align}
Introduce the $d_i\times d_i$ matrix function $\vect{c}_i^{(m)}$ and the $d_j\times d_i$ matrix function $\vect{a}_{ij}^{(m,\ell)}$, $i,j\in \mathcal{J}$, $j\neq i$, and $\ell\in \{1,\ldots,n\}$, given by
\begin{align*}
\mat{c}_i^{(m)}\big(u; \sai_{n}\big) &= 
\sum_{\ell = 1}^{n} \mathds{1}_{\left[t_{\ell-1},\, t_{\ell}\right)}(u)\mathds{1}_{(y_{\ell-1}\, =\, i)}\bar{\mat{P}}^{(m)}_{i}\!\left(u, t_{\ell}\right)
\vect{\alpha}^{(m)}_{\ell}\!\left(\sai_{n}\right)    
\vect{\alpha}^{(m)}\!\left(\sai_{\ell-1}\right)\!
\bar{\mat{P}}^{(m)}_{i}\!\left(t_{\ell-1}, u\right)\!,\\[0.3 cm]
\mat{a}_{ij}^{(m, \ell)}\big(u; \sai_{n}\big) &= 
 \mathds{1}_{\left(t_{\ell-1},\, t_{\ell}\right]}(u)\mathds{1}_{(y_{\ell-1}\, = \, i,\, y_{\ell}\, =\, j)}\bar{\mat{P}}^{(m)}_{j}\!\left(u, t_{\ell+1}\right)\!
\vect{\alpha}^{(m)}_{\ell+1}\!\left(\sai_{n}\right)    
\vect{\alpha}^{(m)}\!\left(\sai_{\ell-1}\right)\!
\bar{\mat{P}}^{(m)}_{i}\!\left(t_{\ell-1}, u\right)\!,
\end{align*}
where the $d_{y_{\ell-1}}$--dimensional row vector
\begin{align*}
\vect{\alpha}_{\ell}^{(m)}(\sai_n) 
&=  
\mat{M}_{y_{\ell-1} y_{\ell}}\big(t_{\ell}; \mat{\theta}^{(m)}\big)
\!\left(\,\prod_{r = \ell }^{n-1}
\bar{\mat{P}}^{(m)}_{y_r}(t_r,t_{r+1})\mat{M}_{y_r y_{r+1}}\big(t_{r+1}; \mat{\theta}^{(m)}\big)\right)\!\vect{1}_{d_{y_n}}
\end{align*}
takes care of sample path probabilities from the $\ell$'th jump and onwards. We then have the following main result. 
\begin{theorem}\label{thm:cond_log-lik_general}
The conditional expected log-likelihood given the macro data $\vect{\mathcal{S}}$ under the parameter  $(\vect{\pi}^{(m)}_1, \mat{\theta}^{(m)})$, $m\in \mathbb{N}_0$, is given by
\begin{align}\nonumber
\bar{L}^{(m)}(\vect{\pi}_1, \mat{\theta}) &=
\bar{L}_0^{(m)}(\vect{\pi}_1)  +\sum_{\mat{\iay},\mat{\jay}\in E \atop \mat{\jay}\neq \mat{\iay} } \bar{L}^{(m)}_{\mat{\iay}\mat{\jay}}(\mat{\theta}), \\[0.2 cm]\label{eq:cond_log-lik_general}
\bar{L}_0^{(m)}(\vect{\pi}_1) &= \sum_{r = 1}^{d_1}\bar{B}^{(m)}_{(1,r)}(0) \log(\pi_{(1,r)}(0)), \\[0.2 cm]\nonumber
\bar{L}^{(m)}_{\mat{\iay}\mat{\jay}}(\mat{\theta}) &= \int_{(0,T]} \log\!\left(\mu_{\mat{\iay}\mat{\jay} }(u;\bm{\theta})\right)\!\!\md \bar{N}^{(m)}_{\mat{\iay}\mat{\jay}}\!\left(u\right)  - \int_{0}^{T } \bar{I}_{\mat{\iay}}^{(m)}\!\left(u\right)\! \mu_{\mat{\iay}\mat{\jay}}(u;\bm{\theta})\md u,
\end{align}
where, for $r\in \{1,\ldots,d_1\}$,
\begin{align*}
\bar{B}^{(m)}_{(1,r)}(0)  &= \sum_{n=1}^N \frac{  
\pi_{(1,r)}^{(m)}(0)\, \vect{e}^{\prime}_{r}
\bar{\mat{P}}^{(m)}_1\!\big(0, T_1^{(n)}\big)
\vect{\alpha}^{(m)}_{1}\!\big(\mathcal{S}^{(n)}\big)
}{
\vect{\alpha}^{(m)}\!\big(\mathcal{S}^{(n)}\big)\vect{1}_{n}
},
\end{align*}
while for $\mat{\iay}\in E$ and $\check{i}\in \{1,\ldots,d_i\}$, $\check{i}\neq \tilde{i}$,
\begin{align*}
\bar{I}_{\mat{\iay}}^{(m)}(u) &=
\sum_{n=1}^N \frac{\vect{e}_{\tilde{i}}^{\prime}\mat{c}_i^{(m)}\!\!\left(u; \mathcal{S}^{(n)}\right)\!\vect{e}_{\tilde{i}}
}{
\vect{\alpha}^{(m)}\!\big(\mathcal{S}^{(n)}\big)\vect{1}_{n}
},
\\[0.6 cm]
\md \bar{N}_{\mat{\iay}(i,\check{i})}^{(m)}(u) 
&=
\sum_{n=1}^N \mu_{\mat{\iay}(i,\check{i})}\big(u; \mat{\theta}^{(m)}\big)\frac{\vect{e}_{\check{i}}^{\prime}\mat{c}_i^{(m)}\!\!\left(u; \mathcal{S}^{(n)}\right)\!\vect{e}_{\tilde{i}}
}{
\vect{\alpha}^{(m)}\!\big(\mathcal{S}^{(n)}\big)\vect{1}_{n}
}
\md u, 
\end{align*}
and for $\mat{\jay}\in E$, $j\neq i$,
\begin{align*}
&\md \bar{N}_{\mat{\iay}\mat{\jay}}^{( m)}\!\left(u\right)=\sum_{n=1}^N \sum_{\ell = 1}^{M^{(n)}}
\mu_{\mat{\iay}\mat{\jay}}\big(u; \mat{\theta}^{(m)}\big)
\frac{  
\vect{e}_{\tilde{j}}^{\prime}\mat{a}_{ij}^{(m, \ell)}\!\big(u; \mathcal{S}^{(n)}\big)\vect{e}_{\tilde{i}}
}{
\vect{\alpha}^{(m)}\!\big(\mathcal{S}^{(n)}\big)\vect{1}_{n}
}
\md \varepsilon_{T^{(n)}_\ell}(u),
\end{align*}
where $\varepsilon_{T^{(n)}_\ell}$ is the Dirac measure in $T^{(n)}_\ell$. 
\end{theorem}
\begin{proof}
See Appendix \ref{ap:proofs}.
\end{proof}
\subsection{EM algorithm within the reset property}\label{subsec:EM_reset} 
We now assume that the reset property on the form \eqref{eq:reset_parameter_general} is satisfied, such that the complete data log--likelihood takes the form \eqref{eq:log-likelihood_reset}.\ Since we in this case parametrize transition rates in $(\mat{\theta}, \mat{\eta})\in \mat{\Theta}\times \mat{H}$, the conditional expected log--likelihood given the macro data $\vect{\mathcal{S}}$, under some current parameter $(\mat{\theta}^{(m)}, \vect{\eta}^{(m)})$ is defined as the map
\begin{align}\label{eq:cond_log-likelihood_reset_general_def}
(\mat{\theta}, \vect{\eta}) \mapsto \bar{L}^{(m)}(\mat{\theta}, \vect{\eta}) = \mathbb{E}_{(\mat{\theta}^{(m)}, \vect{\eta}^{(m)})}\!\!\left[\left.  L^{\mat{X}}(\mat{\theta}, \vect{\eta}) \, \right|  \vect{\mathcal{S}} \right]\!, \quad m\in \mathbb{N}_0,       
\end{align}
where $\mathbb{E}_{(\mat{\theta}, \vect{\eta})}$ denotes the expectation under which the Markov jump process $\mat{X}$ of macro--micro states has transition intensity matrix function $\mat{M}(\cdot ; \mat{\theta}, \vect{\eta})$ and initial distribution $(\vect{\pi}_1(0; \vect{\eta}), \vect{0})$.  

The nature of the reset property allows us to consider each observed macro sojourn independently, and we shall, therefore, group data into the different macrostates so that computations can be carried out locally within macrostates without the need to take care of past and future macro paths.  This is made precise as follows.  For $i\in \mathcal{J}$,  let $$M_i = \sum_{n=1}^N\sum_{\ell = 1}^{M^{(n)}} \mathds{1}_{(Y_{\ell-1}^{(n)} = i)}$$ denote the number of sojourns  in macrostate $i$ observed, and let $\vect{\mathcal{T}}_i = \big(\mathcal{T}_i^{(1)},\ldots,\mathcal{T}_i^{(M_i)}\big)$ be the set of macrostate $i$ observations given by
\begin{align}
\begin{split}\label{eq:Ti_def}
\vect{\mathcal{T}}_i 
 &= \left\{\left. \left(T^{(n)}_{\ell-1},\, Y^{(n)}_{\ell-1},\, T_\ell^{(n)},\, Y_\ell^{(n)}\right)\, \right|\, n=1,\ldots,N, \, \ell = 1,\ldots,M^{(n)}\ \text{s.t} \ Y_{\ell-1}^{(n)} = i\right\}\! \\[0.4 cm]
&= \left\{\left(R_i^{(n)},\, i,\, \tau_i^{(n)}, Z_i^{(n)}\right)\right\}_{n\in \{1,\ldots,M_i\}}.
\end{split}
\end{align}
Then $\vect{\mathcal{T}}_i$ contains data points for macrostate $i$, consisting of time of entries $R_i^{(n)}$ into the macrostate, jump times $\tau_i^{(n)}$ out of the state, and macrostates $Z_i^{(n)}$ jumped to at time $\tau_i^{(n)}$.\ Similar type of data representations are made in \cite{breuer2002algorithm}. 

For a generic realisation $\mathcal{t}_i = (r_i,i,\tau_i,z_i)$ of $\mathcal{T}_i^{(n)}$, the matrix function $\vect{c}^{(m)}_i$ and defective distribution $\vect{\alpha}^{(m)}$ now satisfy, for $u\in (r_i, \tau_i]$,
\begin{align}\label{eq:ci_reset_general}
\begin{split}
\mat{c}_i^{(m)}(u; \mathcal{t}_i) &=  \bar{\mat{P}}^{(m)}_{i}(u, \tau_i)
\vect{\beta}^{(m)}_{iz_i}(\tau_i)    
\vect{\pi}_i^{(m)}(r_i)
\bar{\mat{P}}^{(m)}_{i}(r_i, u),\\[0.2 cm]
\vect{\alpha}^{(m)}(\mathcal{t}_i)\vect{1}_{d_{z_i}} &= \vect{\pi}_i^{(m)}(r_i)
\bar{\mat{P}}^{(m)}_{i}(r_i, \tau_i)\vect{\beta}^{(m)}_{iz_i}(\tau_i).
\end{split}
\end{align} 
Concerning jumps between macrostates, introduce the aggregated conditional expected statistics
\begin{align*}
\bar{N}_{\mat{\iay}j}(s) = \mathbb{E}_{(\mat{\theta}^{(m)}, \vect{\eta}^{(m)})}\!\!\left[\left.  N_{\mat{\iay}j}(s) \, \right|  \vect{\mathcal{S}} \right] 
\qquad   \text{and} \qquad 
\bar{N}_{\mat{\iay}}(s) &= \mathbb{E}_{(\mat{\theta}^{(m)}, \vect{\eta}^{(m)})}\!\left[\left.  N_{\mat{\iay}}(s) \, \right|  \vect{\mathcal{S}} \right]\!, 
\end{align*}
where the aggregated statistics $N_{\mat{\iay}j}$ and $N_{\mat{\iay}}$ are given in \eqref{eq:agg_N_reset_general}.\ Furthermore, the conditional expected statistics within macrostates, $\bar{I}^{(m)}_{\mat{\iay}}$ and $\bar{N}^{(m)}_{\mat{\iay}(i,\check{i})}$, are given as in the general case in \eqref{eq:cond_statistics_general_I}--\eqref{eq:cond_statistics_general_N}, but where the expectation is taken under $(\mat{\theta}^{(m)}, \vect{\eta}^{(m)})$, i.e.\ with the operator $\mathbb{E}_{(\mat{\theta}^{(m)}, \vect{\eta}^{(m)})}$.\  We now have the following result. 
\begin{theorem}\label{cor:cond_log-likelihood_reset_general}
Suppose that the reset property \eqref{eq:indep_cond} holds. Then the conditional expected log-likelihood \eqref{eq:cond_log-likelihood_reset_general_def} is given by
\begin{align}\label{eq:cond_log-likelihood_reset_general}
\bar{L}^{(m)}(\mat{\theta}, \vect{\eta}) 
&= 
\sum_{\mat{\iay}\in E}\Bigg(
\bar{L}^{(m)}_{\mat{\iay}}(\vect{\eta}) 
+ \sum_{\check{i} = 1\atop \check{i}\neq \tilde{i}}^{d_i}\bar{L}^{(m)}_{\mat{\iay}(i,\check{i})}(\mat{\theta}) 
+ \sum_{j\in \mathcal{J}\atop j\neq i}\bar{L}^{(m)}_{\mat{\iay}j}(\mat{\theta})   
\Bigg), 
\end{align}
where $\bar{L}^{(m)}_{\mat{\iay}(i,\check{i})}$ is as in \eqref{eq:cond_log-lik_general}, while 
\begin{align}
\begin{split}\label{eq:cond_log-likelihood_reset_general_init_jump}
\bar{L}^{(m)}_{\mat{\iay}}(\vect{\eta})
&=
 \int_{[0,T]} \log\left(\pi_{\mat{\iay}}(u ; \vect{\eta})\right)\!\!\md \bar{N}^{(m)}_{\mat{\iay}}(u), 
\\[0.2 cm]
\bar{L}^{(m)}_{\mat{\iay}j}(\mat{\theta})
&= 
 \int_{(0,T]} \log\left(\beta_{\mat{\iay}j}(u ; \mat{\theta})\right)\!\!\md \bar{N}^{(m)}_{\mat{\iay}j}(u)-\int_0^T \bar{I}^{(m)}_{\mat{\iay}}(u) \beta_{\mat{\iay}j}(u ; \mat{\theta})\md u,
\end{split}
\end{align}
but where the conditional expected statistics are given by, for $\mat{\iay}\in E$ and $\check{i}\in \{1,\ldots,d_i\}$, $\check{i} \neq \tilde{i}$,
\begin{align}
\begin{split}\label{eq:cond_statistics_reset_general_1}
\md \bar{N}_{\mat{\iay}}^{(m)}(u) &= \sum_{n=1}^{M_i} 
\frac{  
\pi_{\mat{\iay}}\big(u; \vect{\eta}^{(m)}\big)\vect{e}^{\prime}_{\tilde{i}}
\bar{\mat{P}}^{(m)}_i\big(u, \tau_i^{(n)}\big)
\vect{\beta}_{iZ_i^{(n)}}\!\big(\tau_i^{(n)}; \mat{\theta}^{(m)}\big)
}{
\vect{\alpha}^{(m)}\!\big(\mathcal{T}_i^{(n)}\big)\vect{1}_{n}
}
\md \varepsilon_{R_i^{(n)}}(u), \\[0.4 cm]
\bar{I}^{(m)}_{\mat{\iay}}(u) 
&= 
\sum_{n=1}^{M_i} \frac{  
\vect{e}_{\tilde{i}}^{\prime} \mat{c}_i^{(m)}\big(u; \mathcal{T}_i^{(n)}\big)\vect{e}_{\tilde{i}}
}{
\vect{\alpha}^{(m)}\!\big(\mathcal{T}_i^{(n)}\big)\vect{1}_{n}
},\\[0.2 cm]
\md \bar{N}^{(m)}_{\mat{\iay}(i,\check{i})}(u) 
&=  
\sum_{n=1}^{M_i}
\mu_{\mat{\iay}(i,\check{i})}\big(u; \mat{\theta}^{(m)}\big)
\frac{  
\vect{e}_{\check{i}}^{\prime} \mat{c}_i^{(m)}(u; \mathcal{T}^{(n)}_i )\vect{e}_{\tilde{i}}
}{
\vect{\alpha}^{(m)}\!\big(\mathcal{T}_i^{(n)}\big)\vect{1}_{n}
}\md u,
\end{split}
\end{align}
while for $j\in \mathcal{J}$, $j\neq i$,
\begin{align}\label{eq:cond_statistics_reset_general_2}
\md \bar{N}_{\mat{\iay}j}^{(m)}(u) = \sum_{n=1}^{M_i} \mathds{1}_{(Z_i^{(n)} = j)}
\frac{  
\vect{\pi}_{i}\big(R_{i}^{(n)}; \vect{\eta}^{(m)}\big)
\bar{\mat{P}}^{(m)}_i\!\big(R_i^{(n)}, u\big)
\vect{e}_{\tilde{i}}
\beta_{\mat{\iay}j}(u; \mat{\theta}^{(m)})
}{
\vect{\alpha}^{(m)}\!\big(\mathcal{T}_i^{(n)}\big)\vect{1}_{n}
}
\md \varepsilon_{\tau_i^{(n)}}(u). 
\end{align}
\end{theorem}  
\begin{proof}
See Appendix \ref{ap:proofs}.
\end{proof}
\begin{remark}\label{rem:reset_general_IPH}\rm
The conditional expected log-likelihood \eqref{eq:cond_log-likelihood_reset_general} can be seen to have close relations to the conditional expected log-likelihood of \cite[Theorem 3.4]{IPH_Piecewise}.\ Indeed, consider, e.g.\
\begin{align*}
\int_0^T \bar{I}^{(m)}_{\mat{\iay}}(u)\mu_{\mat{\iay}\mat{\jay}}(u; \mat{\theta})\!\md u 
= 
\sum_{n=1}^{M_i}\int_{R_i^{(n)}}^{\tau_i^{(n)}}   
\frac{  
\vect{e}_{\tilde{i}}^{\prime} \mat{c}_i^{(m)}\big(u; \mathcal{T}_i^{(n)}\big)\vect{e}_{\tilde{i}}
}{
\vect{\alpha}^{(m)}\!\big(\mathcal{T}_i^{(n)}\big)\vect{1}_{n}
}\mu_{\mat{\iay}\mat{\jay}}(u; \mat{\theta})\md u.
\end{align*}  
Looking at a single term on the right--hand side and applying the substitution $v = u-R_i^{(n)}$ to the integral, we see that it equals
\begin{align*}
\int_{0}^{\tau_i^{(n)}-R_i^{(n)}}   
\frac{  
\vect{e}_{\tilde{i}}^{\prime}\mat{c}_i^{(m)}\!\!\left(v+R_i^{(n)}; \mathcal{T}_i^{(n)}\right)\!   \vect{e}_{\tilde{i}}
}{
\vect{\alpha}^{(m)}\!\big(\mathcal{T}_i^{(n)}\big)\vect{1}_{n}
}\mu_{\mat{\iay}\mat{\jay}}(v+R_i^{(n)}; \mat{\theta})\md v 
\end{align*}
where the shifted versions of $\mat{c}_i^{(m)}$ equals, with $\vect{\beta}^{(m)}_{iz_i}(\cdot) = \vect{\beta}_{iz_i}(\cdot; \mat{\theta}^{(m)})$, and also for $\vect{\pi}_i$,  
\begin{align*}
\Prodi_v^{\tau_i^{(n)}-R_i^{(n)}}\!\!\!\!\!\!\left(\bm{I}+\mat{M}^{(m)}_{ii}\!\big(x+R_i^{(n)}\big)\!\md x\right)\!
\vect{\beta}^{(m)}_{iz_i}\!\big(\tau_i^{(n)}\big)    
\vect{\pi}_i^{(m)}\!\big(R_i^{(n)}\big)\!
\Prodi_0^{v} \!\left(\bm{I}+\mat{M}^{(m)}_{ii}\!\big(x+R_i^{(n)}\big)\!\md x\right)\!.
\end{align*}
Performing similar type of manipulations for the other terms in \eqref{eq:cond_statistics_reset_general_1}--\eqref{eq:cond_statistics_reset_general_2}, we see that each term in the conditional expected log--likelihood  \eqref{eq:cond_log-likelihood_reset_general}, corresponding to each macro sojourn $n\in \{1,\ldots,M_i\}$, $i\in \mathcal{J}$, equals the conditional expected log--likelihood of \cite[Theorem 3.4]{IPH_Piecewise}, if in the latter we have     
\begin{itemize}
\item A single IPH observation $\tau_i^{(n)}-R_i^{(n)}$\medskip
\item Initial distribution $\vect{\pi}_i\big(R_i^{(n)}; \vect{\eta}^{(m)}\big)$\medskip
\item Sub--intensity matrix function $x\mapsto \mat{M}_{ii}\big(x + R^{(n)}_{i}; \mat{\theta}^{(m)}\big)$\medskip
\item Exit rate vector function $x\mapsto \vect{\beta}_{iZ_i^{(n)}}\!\big(x+R_i^{(n)}; \mat{\theta}^{(m)}\big)$ 
\end{itemize}
on the state-space $\{1,\ldots,d_i\}$ of transient states, and with parameter space $\mat{\Theta}$. \demoo
\end{remark}
It follows from Theorem \ref{cor:cond_log-likelihood_reset_general} and Remark \ref{rem:reset_general_IPH} that the E--step of the EM algorithm for the aggregate Markov model with the reset property can be formulated in terms of, and executed by, the E--step of \cite[Algorithm A]{IPH_Piecewise}.\ The computational demand of performing the estimation procedure, in this case, is therefore comparable to those for general IPHs. As explained in \cite[Subsection 3.3]{IPH_Piecewise}, these computational demands are generally much higher than those of, e.g., \cite{Albrecher-Bladt-Yslas-2020, AsmussenEM}, and assuming piecewise constant transition rates may therefore be of significant advantage. We shall follow this approach for the remainder of the paper to obtain our main algorithm, using the setup made in Subsection \ref{subsec:comple_data_piecewise_reset}. 

\subsection{EM algorithm with piecewise constant transition rates within the reset property}\label{subsec:EM_reset_piecewise}
We now consider the simplifications arising from assuming that the transition intensity matrix function $\mat{M}(\cdot; \mat{\theta}, \vect{\eta})$ is piecewise constant on the form \eqref{eq:rates_piecewise} along with the reset property \eqref{eq:reset_piecewise} being satisfied. Since the resulting complete data log-likelihood, \eqref{eq:log-likelihood_reset} with  \eqref{eq:log-likelihood_piecewise_statewise}  and \eqref{eq:log-likelihood_reset_piecewise_statewise}--\eqref{eq:log-likelihood_reset_piecewise_statewise2}, is linear in the sufficient statistics, we see that for the E--step, it now suffices to compute the conditional expected sufficient statistics,
\begin{align}
\begin{split}\label{eq:cond_stikproever_general}
\bar{B}^{(m)}_{\mat{\iay}}(k) 
&= 
\mathbb{E}_{(\mat{\theta}^{(m)},\, \vect{\eta}^{(m)})}\!\left[\left. B_{\mat{\iay}}(k) \, \right|  \vect{\mathcal{S}} \right]\!, \\[0.3 cm]
\bar{E}^{(m)}_{\mat{\iay}}(k) &= 
\mathbb{E}_{(\mat{\theta}^{(m)},\, \vect{\eta}^{(m)})}\!\left[\left. E_{\mat{\iay}}(k) \, \right|  \vect{\mathcal{S}} \right]\!, \\[0.3 cm]
\bar{O}^{(m)}_{\mat{\iay}(i,\check{i}) }(k) &= 
\mathbb{E}_{(\mat{\theta}^{(m)},\, \vect{\eta}^{(m)})}\!\left[\left. O_{\mat{\iay}(i,\check{i})}(k) \, \right|  \vect{\mathcal{S}} \right]\!,\\[0.3 cm]
\bar{O}^{(m)}_{\mat{\iay}j}(k) &= 
\mathbb{E}_{(\mat{\theta}^{(m)},\, \vect{\eta}^{(m)})}\!\left[\left. O_{\mat{\iay}j}(k) \, \right|  \vect{\mathcal{S}} \right]\!, 
\end{split}
\end{align} 
and then the M--step regarding the update of $\mat{\theta}$ simplifies to a Poisson regression, while the update of $\vect{\eta}$ simplifies to a Multinomial regression, as described in Subsection \ref{subsec:comple_data_piecewise_reset}, but where the sufficient statistics are replaced by their conditional expectations computed in the E--step.  

The transition (sub-)probability matrices within macrostates \eqref{eq:barP_theta} and corresponding defective distribution \eqref{eq:density_sn_m} under the parameter $(\mat{\theta}^{(m)}, \vect{\eta}^{(m)})$ are now on the form (cf.\ \eqref{eq:barP_piecewise}--\eqref{eq:density_sn_piecewise}):
\begin{align}\nonumber
\bar{\mat{P}}^{(m)}_i(s,t) &= 
{\rm e}^{\mat{M}_{ii}^{k(s)}(\mat{\theta}^{(m)})
\left(s_{k(s)}-s\right) 
}
\!\left(\,\prod_{\ell = k(s)+1}^{k(t)-1}{\rm e}^{\mat{M}_{ii}^{\ell}(\mat{\theta}^{(m)})(s_{\ell}-s_{\ell-1})} \right)\!
{\rm e}^{\mat{M}^{k(t)}_{ii}(\mat{\theta}^{(m)})
\left(t-s_{k(t)}\right)  			
},
\\[0.2 cm]\label{eq:barP_dens_piecewise_theta}
\vect{\alpha}^{(m)}\big(\mathcal{T}_i^{(n)}\big)\vect{1}_{n} &= \vect{\pi}_i^{k_i^{(n-)} }\!\big(\vect{\eta}^{(m)}\big)
\bar{\mat{P}}^{(m)}_{i}\!\big(R_i^{(n)}, \tau_i^{(n)}\big)\vect{\beta}^{k_i^{(n+)}}_{iZ_i^{(n)}}\!\big(\mat{\theta}^{(m)}\big)
\end{align} 
where we recall that $k(x)$, for $x\geq 0$, equals the unique $k\in \{1,\ldots,K\}$ satisfying that $x\in (s_{k-1},s_k]$; for notational convenience, we put $k^{(n-)}_i = k\big(R_i^{(n)}\big)$ and $k^{(n+)}_i = k\big(\tau_i^{(n)}\big)$. The conditional expected sufficient statistics \eqref{eq:cond_stikproever_general} then follow immediately from the more general results of Theorem \ref{cor:cond_log-likelihood_reset_general}. 
\begin{corollary}\label{cor:cond_statistics_piecewise_reset}
Suppose that the transition intensity matrix function $\mat{M}(\cdot; \mat{\theta}, \vect{\eta})$ is piecewise constant on the form \eqref{eq:rates_piecewise}, and that the reset property \eqref{eq:reset_piecewise} is satisfied.\ Then the conditional expected sufficient statistics \eqref{eq:cond_stikproever_general} are given by, 
\begin{align*}
\bar{B}^{(m)}_{\mat{\iay}}(k)  
&= \sum_{n=1}^{M_i} \mathds{1}_{(k^{(n-)}_i\, = \, k)}
\frac{  
\pi_{\mat{\iay}}^{k^{(n-)}_i}\!\big(\vect{\eta}^{(m)}\big)\, \vect{e}^{\prime}_{\tilde{i}}
\bar{\mat{P}}^{(m)}_i\!\big(R_i^{(n)}, \tau_i^{(n)}\big)
\vect{\beta}^{k^{(n+)}_i}_{iZ_i^{(n)} }\!
\big( \mat{\theta}^{(m)}
\big)
}{
\vect{\alpha}^{(m)}\big(\mathcal{T}_i^{(n)}\big)\vect{1}_{n}
}, 
\\[0.4 cm]
\bar{E}^{(m)}_{\mat{\iay}}(k) 
&=
\sum_{n = 1}^{M_i}
\frac{
\displaystyle\int_{
\tau_{i|k-1}^{(n)}
}^{
\tau_{i|k}^{(n)}
}
\vect{e}_{\tilde{i}}^{\prime}\mat{c}_i^{(m)}\big(u; \mathcal{T}_i^{(n)}\big)\vect{e}_{\tilde{i}}
\md u 
}
{
\vect{\alpha}^{(m)}\big(\mathcal{T}_i^{(n)}\big)\vect{1}_{n}
}
\\[0.4 cm]
\bar{O}^{(m)}_{\mat{\iay}(i,\check{i})}(k) 
&= 
\sum_{n = 1}^{M_i}
\frac{
\displaystyle\int_{\tau^{(n)}_{i | k-1}}^{\tau^{(n)}_{i | k}}
\vect{e}_{\check{i}}^{\prime}\mat{c}_i^{(m)}\big(u; \mathcal{T}_i^{(n)}\big)\vect{e}_{\tilde{i}}\md u
}
{
\vect{\alpha}^{(m)}\big(\mathcal{T}_i^{(n)}\big)\vect{1}_{n}
}\mu_{\mat{\iay}(i,\check{i})}^k\!\big(\mat{\theta}^{(m)}\big)
, 
\\[0.4 cm]
\bar{O}^{(m)}_{\mat{\iay}j}(k) 
&= 
\sum_{n=1}^{M_i} \mathds{1}_{(Z_i^{(n)}\, =\, j)}\mathds{1}_{(k^{(n+)}_i\, = \, k)}
\frac{  
\vect{\pi}_{i}^{k^{(n-)}_i}\!\big(\vect{\eta}^{(m)}\big)
\bar{\mat{P}}^{(m)}_i\!\big(R_i^{(n)}, \tau_i^{(n)}\big)
\vect{e}_{\tilde{i}}
\beta_{\mat{\iay}j}^{k^{(n+)}_i}\!(\mat{\theta}^{(m)})
}{
\vect{\alpha}^{(m)}\big(\mathcal{T}_i^{(n)}\big)\vect{1}_{n}
},
\end{align*}
with $\bar{\mat{P}}^{(m)}$ and $\vect{\alpha}^{(m)}$ as in \eqref{eq:barP_dens_piecewise_theta}, and where $$\tau^{(n)}_{i | k} = \left(s_k\vee r_{i}^{(n)}\right)\! \wedge \tau_i^{(n)}.$$  
\end{corollary}
\begin{proof}
By inserting the expressions for the sufficient statistics in the complete data case, \eqref{eq:OE} combined with \eqref{eq:Occurrences_reset}, into \eqref{eq:cond_stikproever_general}, the result follows immediately from \eqref{eq:cond_statistics_reset_general_1}--\eqref{eq:cond_statistics_reset_general_2} in  Theorem \ref{cor:cond_log-likelihood_reset_general}. 
\end{proof}
By employing the same techniques as in Remark \ref{rem:reset_general_IPH} on the conditional expected statistics of Corollary \ref{cor:cond_statistics_piecewise_reset}, we find that the E--step, in this case, can be written in terms of the E--step of \cite[Algorithm 1]{IPH_Piecewise}, with analogue modifications of shifting all inputs from time $0$ to the time of entries into the macrostate, $R_i^{(n)}$. So the computational demand should be comparable to the estimation of IPHs with piecewise constant transition rates. In particular, the sub--intensity matrix function is shifted, and so one must accordingly shift the grid points on which it is piecewise constant. This is a conceptually different modification from the more general case of Subsection \ref{subsec:EM_reset}.\ 

The complete EM algorithm for the aggregate Markov model with piecewise constant transition rates within the reset property is presented in Algorithm \ref{alg:reset_piecewise} below. We implement this algorithm and show a numerical example of its applicability in {\color{mogens} Section \ref{sec:num}}.  
{\color{mogens}
\begin{remark}[Performance of the EM algorithm]\label{rem:EM_performance}
The EM algorithm is one of a number of possible maximum likelihood methods for incomplete data. It has become the preferred method for (inhomogeneous) phase--type fitting due to its stability compared to competing methods such as direct optimisation or Markov Chain Monte Carlo (MCMC) methods. For IPHs, which are the building blocks of this paper, a comparison to a kernel density method is given in \cite[Section 5.1]{IPH_Piecewise} regarding fitting accuracy and computational burden, which shows promising results. Regarding MCMC methods, there are many problems regarding identifiability, and the estimation of parameters is essentially impossible; see \cite{BladtGonzalesLauritzen} for PHs. \demoo
\end{remark}

\begin{algorithm}[H]
\caption{{EM algorithm for the aggregate Markov model with the reset property and piecewise constant transition rates}}\label{alg:reset_piecewise}
\begin{algorithmic}
\State \textit{\textbf{Input}}: Initial parameters $(\mat{\theta}^{(0)}, \vect{\eta}^{(0)})\in  \mat{\Theta}\times \mat{H}$, and for each macrostate $i\in \mathcal{J}$, data points within the macrostate, 
\begin{align*}
\mat{\mathcal{T}}_i=\left\{\left(r_i^{(n)}, \tau_i^{(n)}, z_i^{(n)}\right)\right\}_{n\in \{1,\ldots, m_i\}},
\end{align*}
consisting of time of entries $r_i^{(n)}$ into the macrostate,  jump times $\tau_i^{(n)}$ out of the state, and macrostate $z_i^{(n)}$ jumped to at time $\tau_i^{(n)}$.\medskip
\begin{enumerate} 
\item[ 0)] Set $m:=0$ 
\item[ 1)]\textit{E-step:}\ For each macrostate $i\in \mathcal{J}$, 
\begin{itemize}\medskip
\item For each sojourn $n\in \{1,\ldots,m_i\}$,\medskip 
\begin{enumerate}[i)]
\item Set $k_i^{(n-)} := k\big(r_i^{(n)}\big)$ and $k_i^{(n+)} := k\big(\tau_i^{(n)}\big)$. \medskip  
\item  Run the E--step of \cite[Algorithm 1]{IPH_Piecewise} with \medskip
\begin{itemize}
\item Grid points $0 = \tilde{s}_0 < \tilde{s}_1 < \cdots < \tilde{s}_{K-k_i^{(n-)} } < \tilde{s}_{K-k_i^{(n-)}+1} = \infty$, where
\begin{align*}
\qquad\qquad\qquad\qquad\qquad\qquad\quad\tilde{s}_k = s_{k_i^{(n-)} + k - 1} - r_i^{(n)}, \qquad k\in \big\{1,\ldots,K-k_i^{(n-)}+1\big\}.
\end{align*}
\item State--space of transient states $\{1,\ldots,d_i\}$\medskip
\item A single IPH observation $\tau_i^{(n)} - r_i^{(n)}$\medskip
\item Initial distribution $\vect{\pi}_i^{k_i^{(n-)} }\!\big(\vect{\eta}^{(m)}\big)$\medskip
\item Sub--intensity matrix function $x\mapsto \mat{M}_{ii}\big(x+r_i^{(n)}; \mat{\theta}^{(m)}\big)$ \medskip
\item Exit rate vector function $x\mapsto \vect{\beta}_{iz_i^{(n)}}\!\big(x+r_i^{(n)}; \mat{\theta}^{(m)}\big)$,\medskip
\end{itemize}
which, across $k\in \big\{1,\ldots,K-k_i^{(n-)}+1\big\}$, gives the output: 
\begin{itemize}\medskip
\item Expected statistics for the initial state:\ $\bar{B}^{(n,m)}_{\mat{\iay}}\big(k_i^{(n-)} \big)$     
\medskip
\item Expected exposures:\ $\bar{E}^{(n,m)}_{\mat{\iay}}\big(k_i^{(n-)}+k-1\big)$
\medskip
\item Expected occurrences between transient states:\ $\bar{O}^{(n,m)}_{\mat{\iay}(i,\check{i})}\big(k_i^{(n-)}+k-1\big)$.   
\item Expected occurrences to the 'absorbing' state: $\bar{O}^{(n,m)}_{\mat{\iay}z_i^{(n)}}\big(k_i^{(n+)} \big)$\medskip
\end{itemize}
\item[iii)]Compute total expected sufficient statistics, for $k\in \{1,\ldots,K\}$,  
\begin{align*}
\qquad\qquad\qquad\bar{B}^{(m)}_{\mat{\iay}}(k) &= \sum_{n=1}^{M_i}\bar{B}^{(n,m)}_{\mat{\iay}}(k),  
&&\bar{O}^{(m)}_{\mat{\iay}(i,\check{i})}(k) = \sum_{n=1}^{M_i}\bar{O}^{(n,m)}_{\mat{\iay}(i,\check{i})}(k),  \\[0.2 cm]
\bar{E}^{(m)}_{\mat{\iay}}(k) &= \sum_{n=1}^{M_i}\bar{E}^{(n,m)}_{\mat{\iay}}(k), 
&&\bar{O}^{(m)}_{\mat{\iay}z_i^{(n)}}(k) = \sum_{n=1}^{M_i}\bar{O}^{(n,m)}_{\mat{\iay}z_i^{(n)}}(k).
\end{align*}
\end{enumerate}
\end{itemize}
\end{enumerate}
\algstore{test_partial}
\end{algorithmic}
\end{algorithm}
\begin{algorithm}
\begin{algorithmic}
\algrestore{test_partial}
\State \begin{enumerate}
\item[ 2)] \textit{M-step:}\ 
Update the parameters:\medskip\medskip
\begin{enumerate}[i)]
\item Compute $\hat{\vect{\eta}}^{(m+1)}$ as the MLE of the  regressions\medskip
\begin{align*}
 \left(\bar{B}^{(m)}_{(i,1 )}(k), \ldots, \bar{B}^{(m)}_{(i,d_i )}(k)\right) \sim \mathrm{Multinomial}\!\left(B_i(k), \bm{\pi}^k_i(\bm{\eta})\right)\!, 
 \end{align*}\par\vspace*{0.3cm}
\ across $k \in \{0,1,\ldots,K\}$ and $i\in \mathcal{J}$. \medskip\medskip
\item Compute $\hat{\mat{\theta}}^{(m+1)}$ as the MLE of the  regressions\medskip\medskip
\begin{align*}
 \bar{O}^{(m)}_{\mat{\iay}(i,\check{i})}(k) &\sim \mathrm{Pois}\Big(\mu^k_{\mat{\iay}(i,\check{i})}(\mat{\theta})\bar{E}^{(m)}_{\mat{\iay}}(k)\Big), \\[0.3 cm]
\bar{O}^{(m)}_{\mat{\iay}j}(k) &\sim \mathrm{Pois}\Big(\beta^k_{\mat{\iay}j}(\mat{\theta})\bar{E}^{(m)}_{\mat{\iay}}(k)\Big). 
\end{align*}\par\vspace*{0.3cm}
\ across $k\in \{1,\ldots,K\}$, $\mat{\iay}\in E$, $\check{i}\in \{1,\ldots,d_i\}$, $\check{i}\neq \tilde{i}$, and $j\in \mathcal{J}$.
\end{enumerate}
  \vspace*{0.2cm}  
\item[3)] Set $m:=m+1$ and GOTO 1), unless a stopping rule is satisfied. \vspace*{0.2cm}
\end{enumerate}
    \State \textit{\textbf{Output}: Fitted parameters $(\hat{\mat{\theta}}, \hat{\vect{\eta}})$.}
\end{algorithmic}
\end{algorithm}

\begin{remark}[On the number of microstates]\label{rem:EM_micro}
The number of microstates in each macrostate is chosen only to improve the fit, particularly to capture duration effects. There are no known methods for deciding the number of microstates other than visual inspection of the fits and comparison of likelihood values. We show numerical examples of this in Section \ref{sec:num}. More systematic methods involving information criteria (like AIC and BIC) are not possible since the minimum number of parameters needed for phase--type representations (of the microstates) is an open and unsolved question. \demoo 
\end{remark}
}


\raggedbottom
\section{Numerical example}\label{sec:num}
In this section, we present a numerical example illustrating the methods developed in Section \ref{sec:em}. The purpose of the example is to let the EM algorithm fit micro intensities based on macro data, which are simulated from a time--inhomogeneous semi--Markov model already used in the context of multi--state life insurance, see, e.g., \cite{hoem72, helwich, christiansen2012, BuchardtMollerSchmidt} for these type of models.\ Since the aggregate Markov model with the reset property admits a time--inhomogeneous semi--Markovian structure (see \cite[Subsection 4.2]{AhmadBladtFurrer}),  we can apply algorithms within this special case to fit an aggregate Markov model with the reset property that sufficiently captures the duration effects appearing in these kinds of models.

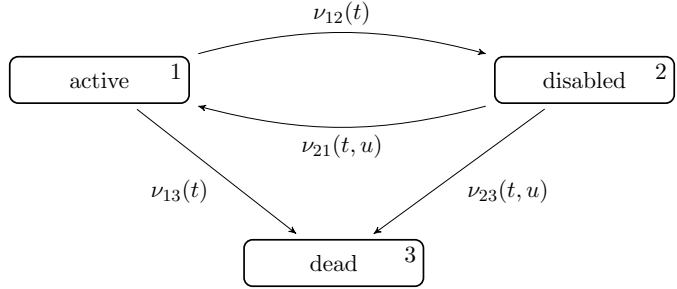
\begin{figure}[H]
	\centering
\scalebox{0.8}{	
\begin{tikzpicture}[node distance=2em and 0em]
		\node[punkt] (2) {disabled};
		\node[anchor=north east, at=(2.north east)]{$2$};
		\node[punkt, left = 50mm of 2] (1) {active};
		\node[anchor=north east, at=(1.north east)]{$1$};
		\node[draw = none, fill = none, left = 25 mm of 2] (test) {};
		\node[punkt, below = 25mm of test] (3) {dead};
		\node[anchor=north east, at=(3.north east)]{$3$};
	\path
		(1)	edge [pil, bend left = 15]	node [above]			{$\nu_{12}(t)$}		(2)
		(2)	edge [pil]				node [below right]	{$\nu_{23}(t,u)$}		(3)
		(1)	edge [pil]				node [below left]		{$\nu_{13}(t)$}		(3)
		(2)	edge [pil, bend left = 15]	node [below]			{$\nu_{21}(t,u)$}		(1)
	;
	\end{tikzpicture}
	}
	\caption{{\color{mogens}Flow diagram of the time-inhomogeneous semi-Markovian disability model with recoveries used to simulate macro data.}}
	\label{fig:disability}
\end{figure} 
We simulate $10,\!000$ paths of macro data from the three--state disability model depicted in Figure \ref{fig:disability}, where we consider a 30--year old male being active upon initiation. For the time and duration--dependent transition rates $\nu_{ij}$,  we use a set of transition rates based on rates employed by a large Danish life insurance company for males, which has been reported to and published by the Danish Financial Supervisory Authority. They are given by, for ages $t\in [30,110]$ and durations $u\leq t$,   
\begin{align}\nonumber
\nu_{13}(\cdot)&\, : \, \text{The 2012 edition of the Danish FSA's longevity benchmark}, \\[0.3 cm]\nonumber
\nu_{12}(t) &=\begin{cases}
\e^{72.53851-10.66927t+0.53371t^2-0.012798t^3 + 1.4922\cdot 10^{-4} t^4 -6.8007\cdot 10^{-7} t^5}		\qquad &\text{for}\ t\leq 67 \\ 
0.0009687435 &\text{for}\  t>67
\end{cases}
\\[0.3 cm] \label{eq:rates_pfa}
\nu_{21}(t,u) &=  
\begin{cases}
\e^{-0.9148875  -0.0309126t +  4.8715347u}\qquad &\text{for}\ u\leq 0.2291667  \\ 
\e^{0.3766531 -0.0309126t -0.7642786u} &\text{for}\  u\in (0.2291667, 2]\\ 
\e^{-0.4808001 -0.0309126t -0.335552u} &\text{for}\  u\in (2, 5]\\ 
\e^{-0.042168 -0.092455t} &\text{for}\  u>5
\end{cases}\\[0.3 cm]  \nonumber
\nu_{23}(t,u) &= 
\begin{cases}
\e^{-6.1057464+0.0635736t -0.2891195u}\qquad &\text{for}\ u\leq 5\\ 
\e^{-11.9169277+0.1356766t} &\text{for}\  u>5
\end{cases}
\end{align}
Since this model only contains duration dependence regarding transitions from the disabled {\color{mogens} state}, we fit the aggregate Markov model depicted in Figure \ref{fig:disability_micro} to the simulated data, where only the disabled state contains {\color{mogens} microstates} so that $d_1 = d_3 = 1$ and $d_2\geq 1$. This model always satisfies the reset property, allowing us to capture the duration effects regarding transitions from the disabled state. 

\begin{figure}[h]
	\centering
\scalebox{0.8}{	\begin{tikzpicture}[node distance=2em and 0em]
		\node[punkt] (21) {$(2,\!1)$};
		\node[right = 15mm of 21] (2temp) {$\cdots$};
		\node[punkt, right = 38mm of 21] (22) {$(2,d_2)$};
		\node[punkt, left = 30mm of 21] (11) {active $(1,\!1)$};
		\node[draw = none, fill = none, left = 50 mm of 22] (test) {};
		\node[punkt, below = 30mm of test] (3) {dead $(3,\!1)$};
		\node[right = 15pt] at ($(22.south)+(0,1.2)$) {disabled};
		\draw[thick, dotted] ($(21.north west)+(-0.5,0.7)$) rectangle ($(22.south east)+(0.8,-0.5)$);
		\path
		($(11.north east)$)	edge [pil, bend left = 15]	($(21.north west) + (-0.6,-0.1)$)
		($(21.south west)+(-0.6,0.1)$) edge [pil, bend left=15] ($(11.south east)$)
		($(22.south) + (-1.7,-0.6)$) edge [pil, bend left=15] ($(3.north east)$)
		($(21.north east)  + (-0.1,0.1)$)	edge [pildotted, bend left = 10] ($(22.north west) + (0.1,0.1)$)
		($(22.south west)  + (0.1,-0.1)$)	edge [pildotted, bend left = 10] ($(21.south east) + (-0.1,-0.1)$)
		($(21.east)  + (0.1,0.1)$)	edge [pildotted, bend left = 10] ($(2temp.west) + (-0.1,0.1)$)
		($(2temp.west)  + (-0.1,-0.1)$)	edge [pildotted, bend left = 10] ($(21.east) + (0.1,-0.1)$)
		($(2temp.east)  + (0.1,0.1)$)	edge [pildotted, bend left = 10] ($(22.west) + (-0.1,0.1)$)
		($(22.west)  + (-0.1,-0.1)$)	edge [pildotted, bend left = 10] ($(2temp.east) + (0.1,-0.1)$)
		($(11.south)$) edge [pil, bend right=15] ($(3.north west)$)
	;
	\end{tikzpicture}}
	\caption{Disability model with $d_2$ unobservable disability microstates.}
	\label{fig:disability_micro}
\end{figure}
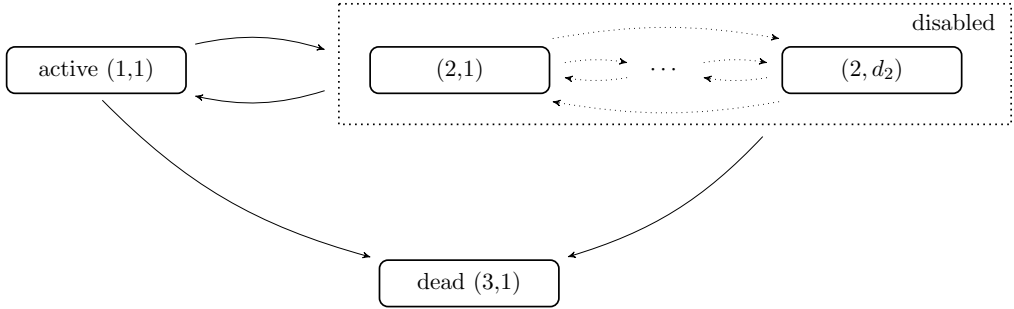
Algorithm \ref{alg:reset_piecewise} is then used to fit piecewise constant transition rates and initial distributions on the micro level in the disabled state for a varying number of disability microstates, $d_2\in \{1,2,3,5,7,10\}$.\ We use piecewise constant transition rates on the grid $\{30,36,37,\ldots,88,110\}$, {\color{mogens} which consists of yearly points with larger gaps at the beginning and end due to lack of data at these areas.}.  The rates between microstates as well as exit rates are parametrised similarly as the true rates \eqref{eq:rates_pfa}, that is,  \eqref{eq:pois-regr} with $q=1$ and $g^{(1)}(s) = s$, such that Poisson GLMs of occurrences against time (with logarithmic link function and log--exposure as offset) is carried {\color{mogens} out} in the M--step; this implies that logarithmic transition rates are (piecewise) linear in time. For the initial distribution, we use \eqref{eq:mult-regr} so that Multinomial logistic regressions of the number of initiations against time are carried out in the M--step. {\color{mogens} The total number of parameters in the model then becomes equal to $2(d_2-1)+2d_2(d_2-1)+4d_2$, each term respectively being the number of parameters in the initial distribution $\vect{\pi}_2$, sub-intensity matrix function $\mat{M}_{22}$, and exit rate vector functions $\vect{\beta}_{21},\vect{\beta}_{23}$.\ The estimated parameters obtained from the fits can be found in Table \ref{tab:pars_transpose} for the case $d_2 = 10$, while the final log-likelihood values (for all $d_2$) are shown in Table \ref{tab:log-liks}. The latter shows an overall improvement in the model fits as $d_2$ increases.  
}

\begin{table}[h]
\centering
\resizebox{\columnwidth}{!}{{\color{mogens}
\begin{tabular}{r|rrrrrrrrrr}
  $\tilde{i}$ & 1 & 2 & 3 & 4 & 5 & 6 & 7 & 8 & 9 & 10 \\ 
   \hline
$\eta_{(2,\tilde{i})}^0$ &  & -0.01 & -0.01 & -7.80 & -0.01 & 14.55 & -8.78 & -0.01 & -18.72 & -0.01 \\ 
  $\eta_{(2,\tilde{i})}^1$ &  & -0.77 & -0.77 & 0.08 & -0.77 & -0.42 & 0.10 & -0.77 & 0.18 & -0.77 \\ 
  $\theta_{(2,\tilde{i})(2,1)}^0$ &  & 255.74 & 125.53 & 722.59 & 721.38 & 824.68 & 719.31 & 647.66 & 117.61 & -19.40 \\ 
  $\theta_{(2,\tilde{i})(2,1)}^1$ &  & -7.73 & -4.41 & -20.56 & -21.02 & -23.57 & -20.34 & -18.78 & -3.10 & -0.16 \\ 
  $\theta_{(2,\tilde{i})(2,2)}^0$ & 1734.69 &  & -6.31 & 6.33 & -22.31 & -1.53 & 6.61 & -0.20 & -16.68 & 6.84 \\ 
  $\theta_{(2,\tilde{i})(2,2)}^1$ & -48.18 &  & -0.08 & -0.17 & 0.33 & -0.01 & -0.16 & 0.01 & -0.08 & -0.15 \\ 
  $\theta_{(2,\tilde{i})(2,3)}^0$ & -30.77 & 2.12 &  & 11.69 & -0.43 & 1.22 & 12.15 & 16.01 & 3.04 & 15.59 \\ 
  $\theta_{(2,\tilde{i})(2,3)}^1$ & -0.03 & -0.11 &  & -0.20 & -0.02 & -0.06 & -0.23 & -0.35 & -0.03 & -0.31 \\ 
  $\theta_{(2,\tilde{i})(2,4)}^0$ & -13.36 & -7.30 & 3.54 &  & -3.22 & -55.00 & 1.46 & 7.89 & 1.39 & 4.03 \\ 
  $\theta_{(2,\tilde{i})(2,4)}^1$ & -0.08 & -0.04 & -0.08 &  & -0.09 & 0.65 & -0.02 & -0.28 & -0.03 & -0.15 \\ 
  $\theta_{(2,\tilde{i})(2,5)}^0$ & -28.02 & 1.26 & 9.40 & 11.58 &  & 2.73 & 13.57 & 1.42 & -16.52 & 11.70 \\ 
  $\theta_{(2,\tilde{i})(2,5)}^1$ & -0.08 & -0.04 & -0.20 & -0.19 &  & -0.07 & -0.23 & -0.06 & -0.08 & -0.21 \\ 
  $\theta_{(2,\tilde{i})(2,6)}^0$ & -25.55 & 2.31 & 6.50 & 10.80 & -4.46 &  & 2.10 & 5.78 & 1500.97 & 10.61 \\ 
  $\theta_{(2,\tilde{i})(2,6)}^1$ & -0.01 & -0.07 & -0.17 & -0.17 & 0.05 &  & -0.02 & -0.12 & -41.46 & -0.25 \\ 
  $\theta_{(2,\tilde{i})(2,7)}^0$ & -33.88 & -11.55 & 10.77 & 8.82 & -11.69 & -4.37 &  & 0.78 & -11.49 & 7.58 \\ 
  $\theta_{(2,\tilde{i})(2,7)}^1$ & -0.01 & 0.14 & -0.24 & -0.17 & 0.15 & 0.07 &  & -0.07 & -0.05 & -0.18 \\ 
  $\theta_{(2,\tilde{i})(2,8)}^0$ & -31.72 & 4.16 & 7.07 & 10.11 & -2.99 & 6.29 & 13.86 &  & -17.03 & -0.80 \\ 
  $\theta_{(2,\tilde{i})(2,8)}^1$ & -0.10 & -0.09 & -0.17 & -0.20 & 0.04 & -0.12 & -0.25 &  & -0.08 & 0.01 \\ 
  $\theta_{(2,\tilde{i})(2,9)}^0$ & -0.97 & -6.49 & 105.74 & -5.73 & 112.41 & -10.93 & -4.30 & -12.63 &  & 4.64 \\ 
  $\theta_{(2,\tilde{i})(2,9)}^1$ & 0.01 & -0.11 & -2.82 & -0.18 & -3.28 & 0.08 & -0.14 & -0.06 &  & -0.13 \\ 
  $\theta_{(2,\tilde{i})(2,10)}^0$ & -29.25 & 4.40 & -0.81 & 8.26 & -6.18 & 1.08 & 9.62 & 2.08 & -13.22 &  \\ 
  $\theta_{(2,\tilde{i})(2,10)}^1$ & -0.07 & -0.12 & 0.01 & -0.20 & 0.09 & -0.06 & -0.22 & -0.12 & -0.06 &  \\ 
  $\theta_{(2,\tilde{i})1}^0$ & -0.34 & 19.56 & -40.84 & 2089.82 & -6.23 & 1799.33 & 323.73 & -23.89 & 3.89 & -13.33 \\ 
  $\theta_{(2,\tilde{i})1}^1$ & -0.05 & -0.49 & 0.02 & -58.89 & 0.04 & -51.02 & -9.78 & 0.00 & -0.06 & 0.02 \\ 
  $\theta_{(2,\tilde{i})3}^0$ & -17.03 & -25.91 & -25.49 & 1.18 & -58.45 & -53.46 & 32.13 & -21.98 & 1.38 & -26.46 \\ 
  $\theta_{(2,\tilde{i})3}^1$ & 0.03 & -0.33 & 0.00 & -0.00 & 0.01 & 0.18 & -0.82 & -0.01 & -0.03 & -0.13 \\ 
  \end{tabular}
  }
  }
\caption{{\color{mogens}Estimated parameters $(\hat{\vect{\theta}}, \hat{\vect{\eta}})$ for the case  $d_2=10$. The resulting transition rates on the micro level follows the parametrizations \eqref{eq:pois-regr} and \eqref{eq:mult-regr} with $q=1$ and $g^{(r)}(s) = s^r$, $r\in\{0,1\}$.}} 
\label{tab:pars_transpose}
\end{table}

\begin{table}[h]
\centering{\color{mogens}
\begin{tabular}{c| rrrrrr}
 $d_2$ & 1 & 2 & 3 & 5 & 7 & 10 \\ \hline
 Log-likelihood & -26,315 & -24,324 & -24,241 & -24,153 & -24,132 & -24,115
  \end{tabular}
  }
\caption{
{
{\color{mogens}
Final log-likelihood values for the different number of disability microstates, $d_2$.}
}
} 
\label{tab:log-liks}
\end{table}

We then compare our model fits with the true model, and to avoid statistical noise and allow for comparison to more classic methods, we also carry out a semi-Markovian GLM fit on the macro data, that is, a Poisson GLM fit of occurrences against age and duration with log-link function and log-exposure as offset. Here we use segmented regressions in the duration dimension to obtain the parametrisation \eqref{eq:rates_pfa}.

In Figure \ref{fig:cond_distributions_pfa_rates_s60}, we show for $s = 60.5$ the resulting fits of the conditional survival function and density of the sojourn time distribution in the disabled state given entry into the state at time $s$, respectively given by
\begin{align*}
t \mapsto \vect{\pi}_2(s; \hat{\vect{\eta}})\bar{\mat{P}}_{2}(s,t; \hat{\mat{\theta}})\vect{1}_{d_2} \qquad \text{and} \qquad  t\mapsto \vect{\pi}_2(s; \hat{\vect{\eta}})\bar{\mat{P}}_{2}(s,t; \hat{\mat{\theta}})\!\left(\vect{\beta}_{21}(t; \hat{\mat{\theta}})+\vect{\beta}_{23}(t; \hat{\mat{\theta}}) \right) 
\end{align*}
against their empirical counterparts as well as (cf., e.g., \cite{hoem72,helwich, christiansen2012}):
\begin{align*}
t\mapsto \e^{-\int_s^t \nu_{2\cdot}(v,v-s)\md v} \qquad \text{and} \qquad t\mapsto \e^{-\int_s^t \nu_{2\cdot}(v,v-s)\md v}\!\left(\nu_{21}(t,t-s)+\nu_{23}(t,t-s)\right)
\end{align*}
for the true model and the GLM fits. We see that a single microstate, corresponding to a Markov chain, does not fit very well with the data, which is to be expected as this model is not able to capture the duration effects appearing in the true model. However, we see that by adding microstates to the disabled state, we can fit the distributions with a high accuracy already with $2$-$3$ microstates. A challenging part of the density close to the origin (corresponding to small durations) as well as around discontinuity points requires many more microstates compared to the GLM fit, though, but this does not seem to carry over to the corresponding survival function, where our model fits performs very well.   

\begin{figure}[h]
   \centering
   \includegraphics[width=0.9\textwidth,trim={8mm 4mm 8mm 0mm}]{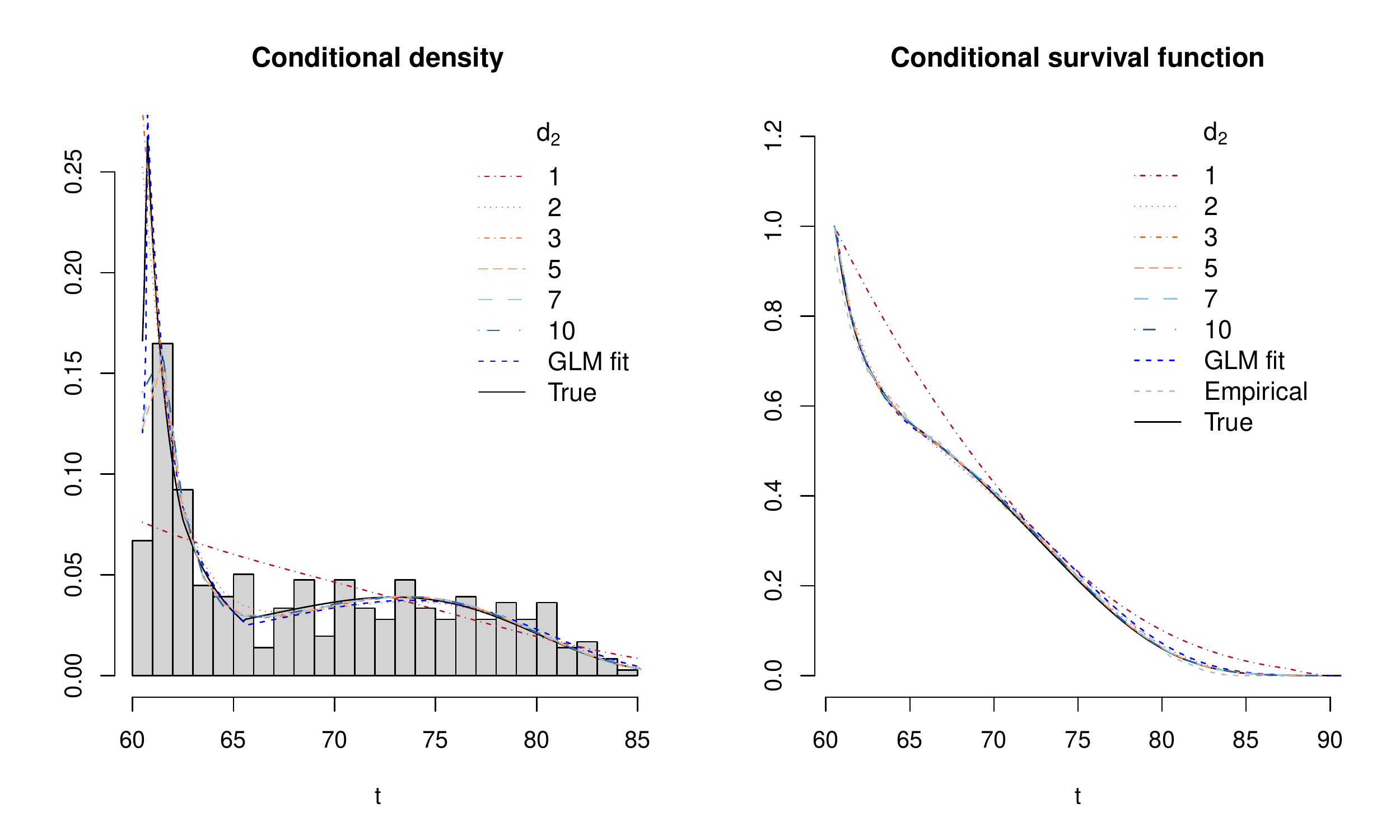}
   \caption{Conditional densities (left) and survival functions (right) of the sojourn time in the disabled state given entry at time $s = 60.5$, for different number of disability microstates, $d_2$, along with the GLM fit as well as true and empirical distributions. The case $d_2 = 1$ corresponds to a Markov chain.}
   \label{fig:cond_distributions_pfa_rates_s60}
\end{figure}

While these survival functions and densities focus on sojourn times and, therefore, not on specific transitions out of the state, we may further examine how the actual transition rates are fitted; these play an important role in multi-state life insurance. Indeed, according to \cite[Subsection 4.2]{AhmadBladtFurrer}, we have fitted an aggregate Markov model with the reset property that admits a time-inhomogeneous semi-Markovian structure with the following time and duration--dependent transition rates out of the disabled state: 
\begin{align}\label{eq:fitted_semi-Markov_rates}
(t,u) \mapsto \frac{
\vect{\pi}_2(t-u; \hat{\vect{\eta}})\bar{\mat{P}}_{2}(t-u,t; \hat{\mat{\theta}})  
}{
\vect{\pi}_2(t-u; \hat{\vect{\eta}})\bar{\mat{P}}_{2}(t-u,t; \hat{\mat{\theta}})\vect{1}_{d_2}
}\vect{\beta}_{2j}(t; \hat{\mat{\theta}}), \qquad j\in \mathcal{J},\ j\neq 2.
\end{align}

In Figure \ref{fig:intensity_pfa_rates_u1}, we examine this as a function of time, $t$, for fixed duration $u = 1$, and compare it with the true rates $t\mapsto \nu_{2j}(t,1)$, as well as its empirical counterpart, namely empirical occurrence--exposure rates on the macro level in the disabled state. Here, we also see how badly the Markov chain case of a single microstate fits. Still, more importantly, we see that it requires more microstates to fit the actual transition rates more accurately than we saw with the sojourn time distributions. However, since we see a higher deviation away from the true model for the GLM fit, especially for recoveries, this might also be due to statistical noise.

\begin{figure}[H]
   \centering
   \includegraphics[width=0.86\textwidth,trim={8mm 4mm 8mm 0mm}]{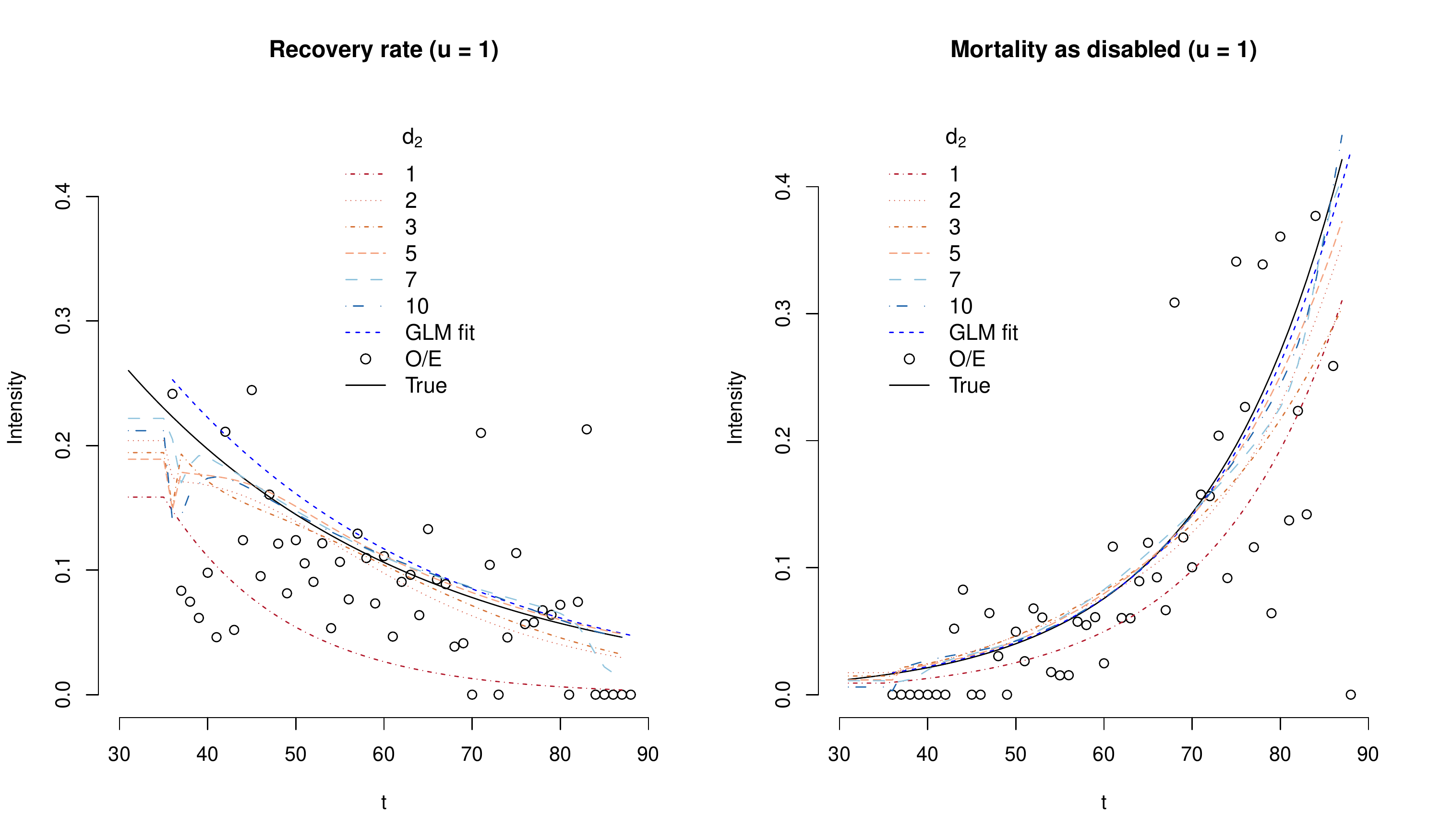}
   \caption{Estimated transition rates \eqref{eq:fitted_semi-Markov_rates} for recoveries (left) and deaths as disabled (right) as a function of time, $t$, with a fixed duration $u = 1$ for different number of disability microstates, $d_2$, along with the GLM fit as well as the true rates, $\nu_{2j}$, and the empirical occurrence--exposure rates.}
   \label{fig:intensity_pfa_rates_u1}
\end{figure}

\begin{figure}[H]
   \centering
   \includegraphics[width=0.86\textwidth,trim={8mm 4mm 8mm 0mm}]{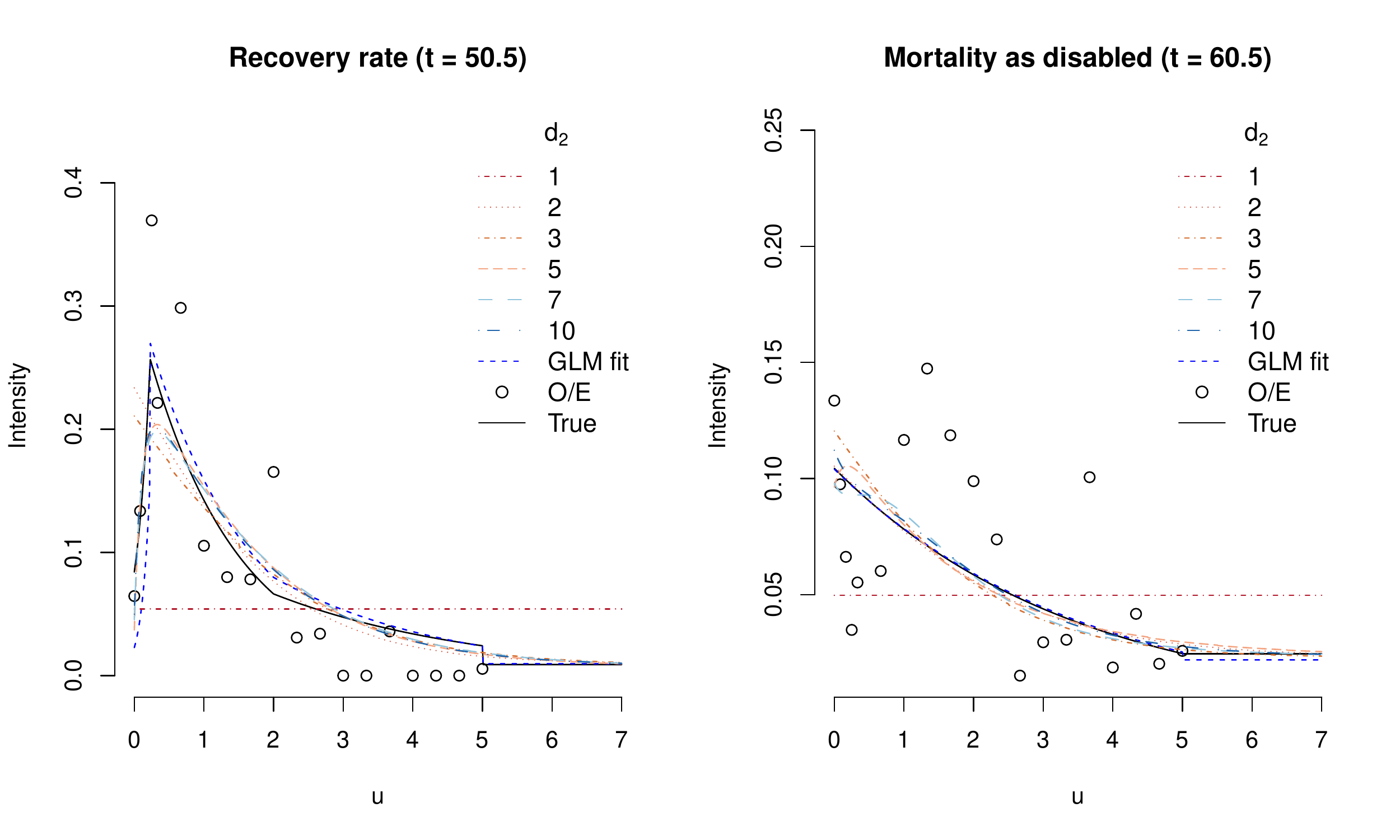}
   \caption{Estimated transition rates \eqref{eq:fitted_semi-Markov_rates} for recoveries (left) and deaths as disabled (right) as a function of duration, $u$, with a fixed time, $t$, for different number of disability microstates, $d_2$, along with the GLM fit as well as the true rates, $\nu_{2j}$, and the empirical occurrence--exposure rates. }
   \label{fig:intensity_pfa_rates_t1}
\end{figure}

Indeed, except for a very small amount of data for recoveries at young ages, we see fits with a high number of microstates that performs at least as well as the GLM fit.   

To round off the analysis, we consider in {\color{mogens}Figure \ref{fig:intensity_pfa_rates_t1}} the fitted rates \eqref{eq:fitted_semi-Markov_rates} as a function of the duration $u$ for a fixed time, $t$.\ Due to the lack of data, we focus on $t = 50.5$ for the recovery rate and $t = 60.5$ for the mortality as disabled, which are ages with sufficient data. Here, we also see the lower accuracy in the fits close to the origin, as with the conditional densities, which is particularly the case for the recovery rate. The irregularity of the true rates in this region seems to require many microstates for the model to capture this kind of duration effect fully. However, the GLM fit also seems to have the same problems, so the performance of the two fits may also be comparable here. For the mortality as disabled, however, we see the GLM fit{\color{mogens}s} much better for durations smaller than five years, while our model then seems to fit better for higher durations. This conforms with the observations made for the conditional densities.

\vspace*{0.5cm}

\noindent \textbf{Acknowledgments.} We thank Christian Furrer for his valuable suggestions and fruitful discussions during this work.


\appendix

\section{}\label{ap:proofs}

{\color{mogens}
For the proof of Theorem \ref{thm:cond_log-lik_general}, we need the following lemma, which builds upon \cite[Lemma 3.2]{IPH_Piecewise}. The result relates to similar conditional distributions considered in \cite{norberg1991, hoem69b}, but where we include past and future jump times in the conditioning.  }

\begin{lemma}\label{lem:help}
Let $\mat{X}=\{\mat{X}(t)\}_{t\geq 0}=\{(X_1(t), X_2(t))\}_{t\geq 0}$ be a time--inhomogeneous Markov jump process on $E$ with transition intensity matrix function $\mat{M}$ and initial distribution $(\vect{\pi}_1(0), \vect{0})$.\ Let $\mathcal{S}_n = (T_i, Y_i)_{i\leq n}$ be the first $n$ jump times and transitions of the macrostate process $X_1$.\ Then, for $\ell\in \{1,\ldots,n\}$, the conditional process within a macro sojourn,   
\begin{align*}
W_\ell(t)  \overset{\mathrm{d}}{=} \mat{X}(t)\big|\mathcal{S}_n \quad \text{on} \quad \left(T_{\ell-1}\leq t <  T_{\ell}\right)\!,
\end{align*}
is a time--inhomogeneous Markov jump process taking values on $\{1,\ldots,d_{Y_ {\ell-1} }\}$ with initial distribution
\begin{align*}
\widetilde{\pi}^\ell_{i}(\mathcal{S}_n) = \frac{\vect{\alpha}(\mathcal{S}_{\ell-1})\vect{e}_{i}\vect{e}_{i}^{\prime}\bar{\mat{P}}_{Y_{\ell-1}}\!(T_{\ell-1}, T_\ell)\vect{\alpha}_{\ell + 1}(\mathcal{S}_n) } {\vect{\alpha}(\mathcal{S}_n)\vect{1}_{d_{Y_n}}  } ,
\end{align*}
transition probabilities
\begin{align*}
\widetilde{p}^\ell_{ij}(t,s\big|\mathcal{S}_n) =   
\frac{
\vect{e}_{i }^{\prime}\bar{\mat{P}}_{Y_{\ell-1}}(t,s)\vect{e}_{j}\vect{e}_{j}^{\prime}\bar{\mat{P}}_{Y_{\ell-1}}(s,T_\ell)\vect{\alpha}_{\ell + 1}(\mathcal{S}_n) 
}{
\vect{e}_{i}^{\prime}\bar{\mat{P}}_{Y_{\ell-1}}(t,T_\ell)\vect{\alpha}_{\ell + 1}(\mathcal{S}_n) 
},
\end{align*}
and transition intensities
\begin{align*}
\widetilde{\mu}_{ij}^\ell(t\big|\mathcal{S}_n) =   
\mu_{(Y_{\ell-1}, i )(Y_{\ell-1}, j )}(t)
\frac{
\vect{e}_{j}^{\prime}\bar{\mat{P}}_{Y_{\ell-1}}(t,T_\ell)\vect{\alpha}_{\ell + 1}(\mathcal{S}_n) 
}{
\vect{e}_{i}^{\prime}\bar{\mat{P}}_{Y_{\ell-1}}(t,T_\ell)\vect{\alpha}_{\ell + 1}(\mathcal{S}_n)
},
\end{align*}
where $\vect{e}_{k}$, $k\in \{1,\ldots,d_{Y_ {\ell-1} }\}$, are $d_{Y_{\ell-1}}$-dimensional column vectors with one in entry $k$ and zeros otherwise.  
\end{lemma}
\begin{proof}
It follows from \cite[Lemma A.1]{AhmadBladtFurrer} that
\begin{align*}
\mathbb{P}\!\left(\left. T_{\ell} \in \md t_{\ell}, Y_{\ell} = y_{\ell}\, \right| \mathcal{S}_{\ell-1} = \sai_{\ell-1}\right) 
= 
\frac{\vect{\alpha}(\sai_{\ell})\vect{1}_{d_{y_{\ell}}}   }{\vect{\alpha}(\sai_{\ell-1})\vect{1}_{d_{y_{\ell-1}}}}\!\md t_{\ell},
\end{align*}
which implies, using that $\mathbb{P}(\mathcal{S}_0 = \sai_0) = 1$,  
\begin{align}\nonumber
\mathbb{P}\!\left( 
T_{n} \in \md t_{n}, Y_{n} = y_{n}, \ldots, 
T_1\in \md t_1, Y_1 = y_1 \right)
&=  
\prod_{\ell=1}^n \mathbb{P}\!\left(\left. T_\ell \in \md t_\ell, Y_\ell  = y_\ell \, \right| \mathcal{S}_{\ell-1} = \sai_{\ell - 1}\right) \\[0.2 cm] \label{eq:density}
&=\prod_{\ell=1}^n \frac{\vect{\alpha}(\sai_{\ell})\vect{1}_{d_{y_{\ell}}}   }{\vect{\alpha}(\sai_{\ell-1})\vect{1}_{d_{y_{\ell-1}}}}\!\md t_{\ell}.  \\[0.2 cm]
&= \vect{\alpha}(\sai_n)\vect{1}_{d_{y_n}}\! \md t_1\md t_2\cdots \md t_n.  \nonumber
\end{align}
The remaining part of the proof now largely follows the approach taken in \cite[Lemma 3.2]{IPH_Piecewise}.\ Fix $\ell \in \{1,\ldots,n\}$ and $t,s\geq 0$.\ Applying same techniques as in \eqref{eq:density}, we get from the Markov property of $\mat{X}$, that on the event $(T_{\ell-1}\leq s<T_\ell)$,
\begin{align}\label{eq:density_shift}
\mathbb{E}\!\left[\left.\mathds{1}_{
(T_\ell \in \md t_\ell, Y_\ell = y_\ell, \ldots, T_n\in \md t_n, Y_n = y_n)} \, 
\right|
\mathcal{F}^{\mat{X}}(s) 
\right] = \vect{e}_{X_2(s)}^{\prime}\bar{\mat{P}}_{X_1(s)}(s,t_\ell)\vect{\alpha}_{\ell+1}(\sai_n)\md t_\ell \cdots\md t_n, 
\end{align}
and similarly when conditioning on $\mathcal{F}^{\mat{X}}(t)$, on the corresponding event at time $t$.\ The transition probabilities for $W_\ell$ are then obtained as follows, on the event $(T_{\ell-1}\leq t < T_\ell)$, 
\begin{align*}
\mathbb{E}\!\left[\left. \mathds{1}_{(W_\ell(s) = \mat{\jay})}\, \right| \mathcal{F}^{W_\ell}(t)\right] &= \mathbb{E}\!\left[\left. \mathds{1}_{(\mat{X} (s) = \mat{\jay})}\mathds{1}_{[T_{\ell-1}, T_\ell)}(s) \,\right| \mathcal{F}^{\mat{X}}(t)\vee \sigma (\mathcal{S}_n)\right] \\[0.2 cm]
&= \mathbb{E}\!\left[\left. \mathds{1}_{(\mat{X} (s) = \mat{\jay})}\mathds{1}_{[T_{\ell-1}, T_\ell)}(s)\,\right| \mathcal{F}^{\mat{X}}(t)\vee \sigma (T_\ell, Y_\ell,\ldots,T_n,Y_n)\right] \\[0.2 cm]
&= \frac{
\mathbb{E}\!\left[\left.  \mathds{1}_{(\mat{X} (s) = \mat{\jay})}\mathds{1}_{[T_{\ell-1}, T_\ell)}(s)\, \right| \mathcal{F}^{\mat{X}}(t)\right]\vect{e}_{\tilde{j}}^{\prime}\bar{\mat{P}}_{j}(s,T_\ell)\vect{\alpha}_{\ell+1}(\mathcal{S}_n)
}{
\vect{e}_{X_2(t)}^{\prime}\bar{\mat{P}}_{X_1(t)}(t,T_\ell)\vect{\alpha}_{\ell+1}(\mathcal{S}_n)
} \\[0.2 cm]
&= \frac{
 \vect{e}_{X_2(t)}^{\prime}\bar{\mat{P}}_{X_1(t)}(t,s)\vect{e}_{\tilde{j}} \vect{e}_{\tilde{j}}^{\prime}\bar{\mat{P}}_{j}(s,T_\ell)\vect{\alpha}_{\ell+1}(\mathcal{S}_n)
}{
\vect{e}_{X_2(t)}^{\prime}\bar{\mat{P}}_{X_1(t)}(t,T_\ell)\vect{\alpha}_{\ell+1}(\mathcal{S}_n)
},
\end{align*}
where use the Markov property of $\mat{X}$ in the second equality, and the tower property with the sigma--algebra $\mathcal{F}^{\mat{X}}(s)\vee \sigma (T_\ell, Y_\ell,\ldots,T_n,Y_n)\supseteq \mathcal{F}^{\mat{X}}(t)\vee \sigma (T_\ell, Y_\ell,\ldots,T_n,Y_n)$ in the third equality. The fourth equality follows from \cite[Lemma A.1 and Proposition 4.1]{AhmadBladtFurrer}. 
Conditioning on $\mat{X}(t) = \mat{\iay}$, we get the desired transition probabilities. From these, the corresponding transition intensities follow immediately: 
\begin{align*}
\widetilde{\mu}_{ij}^\ell(t\big|\mathcal{S}_n) &= \lim_{h\downarrow 0}\frac{\widetilde{p}^\ell_{ij}(t,t+h\big| \mathcal{S}_n)}{h}
\\[0.2 cm]
&= \frac{1}{\vect{e}_{i}^{\prime}\bar{\mat{P}}_{Y_{\ell-1}}(t,T_\ell)\vect{\alpha}_{\ell + 1}(\mathcal{S}_n)} 
\lim_{h\downarrow 0}  \frac{
\vect{e}_{i }^{\prime}\bar{\mat{P}}_{Y_{\ell-1}}(t,t+h)\vect{e}_{j}}{h}\vect{e}_{j}^{\prime}\bar{\mat{P}}_{Y_{\ell-1}}(t+h,T_\ell)\vect{\alpha}_{\ell + 1}(\mathcal{S}_n) \\[0.2 cm]
&=\mu_{(Y_{\ell-1}, i )(Y_{\ell-1}, j )}(t)
\frac{
\vect{e}_{j}^{\prime}\bar{\mat{P}}_{Y_{\ell-1}}(t,T_\ell)\vect{\alpha}_{\ell + 1}(\mathcal{S}_n) 
}{
\vect{e}_{i}^{\prime}\bar{\mat{P}}_{Y_{\ell-1}}(t,T_\ell)\vect{\alpha}_{\ell + 1}(\mathcal{S}_n) 
}. 
\end{align*}
Here, we use the continuity of the transition (sub-)probability matrix function $\bar{\mat{P}}_{Y_{\ell-1}}$ obtained from the continuity of product integrals (whenever they exist). Lastly, to derive the initial distribution, we see using the same techniques as in \eqref{eq:density}, that 
\begin{align}\nonumber
&\mathbb{P}\!\left(T_{n} \in \md t_{n}, Y_{n} = y_{n}, \ldots, T_{\ell-1}\in \md t_{\ell-1}, \mat{X}(T_{\ell-1}) = (y_{\ell-1}, \tilde{i}),\ldots, T_1 \in \md t_1, Y_1 = y_1\right) \\[0.2 cm]\label{eq:lem_initial_dist}
&= \vect{\alpha}(\sai_{\ell-1})\vect{e}_{\tilde{i}}\vect{e}_{\tilde{i}}^{\prime}\bar{\mat{P}}_{y_{\ell-1}}(t_{\ell-1},t_\ell)\vect{\alpha}_{\ell+1}(\sai_n)\md t_1\cdots\md t_n,
\end{align}
and so 
\begin{align*}
\mathbb{P}\!\left(W_\ell(T_{\ell-1}) = \mat{\iay}\right) &= \mathbb{E}\!\left[\left. \mathds{1}_{(\mat{X}(T_{\ell-1}) = \mat{\iay})} \, \right| \sigma(\mathcal{S}_n)\right] = \frac{\vect{\alpha}(\mathcal{S}_{\ell-1})\vect{e}_{\tilde{i}}\vect{e}_{\tilde{i}}^{\prime}\bar{\mat{P}}_{Y_{\ell-1}}(T_{\ell-1},T_\ell)\vect{\alpha}_{\ell+1}(\mathcal{S}_n)   }{\vect{\alpha}(\mathcal{S}_n)\vect{1}_{d_{y_n}}},
\end{align*}   
as desired. 
\end{proof}

\begin{proof}[Proof of Theorem~\ref{thm:cond_log-lik_general}]
The proof largely follows the approach in \cite[Theorem 3.4]{IPH_Piecewise}.\ For notational convenience, we write $\mathbb{E}^{(m)}$ for the expectation operator $\mathbb{E}_{(\vect{\pi}_1^{(m)},\, \mat{\theta}^{(m)})}$. Now, it is evident from an application of Fubini's theorem that the conditional expected log--likelihood \eqref{eq:cond_log-likelihood_general_def} is on the form \eqref{eq:cond_log-lik_general}, and it, therefore, suffices to show the results for the conditional expected statistics \eqref{eq:cond_statistics_general}--\eqref{eq:cond_statistics_general_N}. By independence between the data points in $\vect{\mathcal{S}}$ and Lemma \ref{lem:help}, we get
\begin{align*}
\bar{B}^{(m)}_{(1,r)}(0) = \sum_{n = 1}^N  \mathbb{E}^{(m)}\!\!\left[\left. \mathds{1}_{(\mat{X}^{(n)}(0)\, =\, (1,r))} \, \right|  \mathcal{S}^{(n)} \right] = \tilde{\pi}^0_r\big(\mathcal{S}^{(n)} \big),
\end{align*}
which by insertion yields the desired result. For $\bar{I}_{\mat{\iay}}^{(m)}$, we get
\begin{align*}
\bar{I}_{\mat{\iay}}^{(m)}(u) &=  \sum_{n=1}^N \mathbb{E}^{(m)}\!\!\left[\left. \mathds{1}_{(\mat{X}^{(n)}(u)\, =\, \mat{\iay} )} \, \right|  \mathcal{S}^{(n)} \right] 
\\[0.2 cm]
&= \sum_{n=1}^N\sum_{\ell = 1}^{M^{(n)}}
\mathds{1}_{\big[T_{\ell-1}^{(n)},\, T_\ell^{(n)}\big)}\!(u)
\mathds{1}_{\big(Y_{\ell-1}^{(n)}\, =\, i\big)}
\sum_{\check{i} = 1}^{d_i} \widetilde{\pi}^\ell_{\check{i}}\big(\mathcal{S}^{(n)}\big)\widetilde{p}_{\check{i}\,\tilde{i}}\big(T^{(n)}_{\ell-1},u|\mathcal{S}^{(n)}\big).
\end{align*}
The inner sum is then given by
\begin{align*}
\sum_{\check{i} = 1}^{d_i} \widetilde{\pi}^\ell_{\check{i}}\!\big(\mathcal{S}^{(n)}\big)\widetilde{p}_{\check{i}\tilde{i}}\big(T^{(n)}_{\ell-1},u|\mathcal{S}^{(n)}\big) 
\!&=\!\!
\sum_{\check{i} = 1}^{d_i}\!
 \frac{
\vect{\alpha}^{(m)}\!\big(\mathcal{S}^{(n)}_{\ell-1}\big)
\vect{e}_{\check{i}}\vect{e}_{\check{i}}^{\prime}\!
\bar{\mat{P}}^{(m)}_i\!\big(T^{(n)}_{\ell-1},s)\vect{e}_{\tilde{i}}\vect{e}_{\tilde{i}}^{\prime}\!\bar{\mat{P}}^{(m)}_i\!\big(s,T_\ell^{(n)}\big)\vect{\alpha}^{(m)}_{\ell + 1}\!\big(\mathcal{S}^{(n)}\big) 
}
{
\vect{\alpha}^{(m)}\!\big(\mathcal{S}^{(n)}\big)\vect{1}_{d_{y_n}}  
}\\[0.2 cm]
&= \frac{
\vect{\alpha}^{(m)}\!\big(\mathcal{S}^{(n)}_{\ell-1}\big)
\bar{\mat{P}}^{(m)}_i\!\big(T_{\ell-1}^{(n)},s\big)
\vect{e}_{\tilde{i}}\vect{e}_{\tilde{i}}^{\prime}
\bar{\mat{P}}^{(m)}_i\!\big(s,T_\ell^{(n)}\big)
\vect{\alpha}^{(m)}_{\ell + 1}\!\big(\mathcal{S}^{(n)}\big) 
}
{
\vect{\alpha}^{(m)}\!\big(\mathcal{S}^{(n)}\big)\vect{1}_{d_{y_n}}  
},
\end{align*}
which shows the result for $\bar{I}^{(m)}_{\mat{\iay}}$.\ For $\bar{N}_{\mat{\iay}(i,\check{i})}^{(m)}$, $\check{i}\in \{1,\ldots,d_i\}$, $\check{i}\neq \tilde{i}$, we get 
\begin{align*}
\bar{N}^{(m)}_{\mat{\iay}(i,\check{i})}(u) &= 
\sum_{n=1}^N\sum_{\ell = 1}^{M^{(n)}} \mathbb{E}^{(m)}\!\!\left[\left. \int_{(0,u]}\mathds{1}_{\big[T_{\ell-1}^{(n)},\, T_\ell^{(n)}\big)}\!(x)\mathds{1}_{\big(Y_{\ell-1}^{(n)}\, =\, i\big)} \md N_{\mat{\iay}(i,\check{i})}^{(n)}(x) \, \right|  \mathcal{S}^{(n)} \right]\!. 
\end{align*} 
Then, using the intensity process of $\{\mat{X}^{(n)}(x)\}_{x\in \big[T_{\ell-1}^{(n)},\, T_{\ell}^{(n)}\big)}\big|\mathcal{S}^{(n)}$ from Lemma \ref{lem:help},  
\begin{align*}
\bar{N}^{(m)}_{\mat{\iay}(i,\check{i})}(u) &= \sum_{n=1}^N\!\sum_{\ell = 1}^{M^{(n)}} \mathbb{E}^{(m)}\!\!\left[ \int_0^u \mathds{1}_{\big[T_{\ell-1}^{(n)},\, T_\ell^{(n)}\big)}\!(x)\mathds{1}_{\big(Y_{\ell-1}^{(n)}\, =\, i\big)}\!  \mathds{1}_{(\mat{X}^{(n)}(x)\, =\, \mat{\iay})}
\widetilde{\mu}^\ell_{\tilde{i}\check{i}}\big(x\big|\mathcal{S}^{(n)}\big)\!\md x   \, \bigg|  \mathcal{S}^{(n)} \right] 
\\[0.2 cm]
& = \sum_{n=1}^N\!\sum_{\ell = 1}^{M^{(n)}}  
\!\int_0^u \mathds{1}_{\big[T_{\ell-1}^{(n)},\, T_\ell^{(n)}\big)}\!(x)\mathds{1}_{\big(Y_{\ell-1}^{(n)}\, =\, i\big)}\times \\
&\qquad\qquad\qquad\sum_{r = 1}^{d_i} \widetilde{\pi}^\ell_{r}(\mathcal{S}^{(n)})\widetilde{p}_{r\tilde{i}}\big(T^{(n)}_{\ell-1},x|\mathcal{S}^{(n)}\big)   
\widetilde{\mu}^\ell_{\tilde{i}\check{i}}(x\big|\mathcal{S}^{(n)})\md x  
\\[0.2 cm]
&=\sum_{n=1}^N\!\sum_{\ell = 1}^{M^{(n)}}  
\int_0^u \mathds{1}_{\big[T_{\ell-1}^{(n)},\, T_\ell^{(n)}\big)}(x)\mathds{1}_{\big(Y_{\ell-1}^{(n)}\, =\, i\big)}\times \\[0.2 cm]
&\qquad 
\frac{
\vect{\alpha}^{(m)}\!\big(\mathcal{S}^{(n)}_{\ell-1}\big)\bar{\mat{P}}^{(m)}_i\!\big(T^{(n)}_{\ell-1},x\big)\vect{e}_{\tilde{i}}\mu_{\mat{\iay}(i,\check{i})}\big(x; \mat{\theta}^{(m)} \big)\vect{e}_{\tilde{i}}^{\prime}\bar{\mat{P}}^{(m)}_i\!\big(x,T_\ell^{(n)}\big)\vect{\alpha}^{(m)}_{\ell + 1}\!\big(\mathcal{S}^{(n)}\big) 
}
{
\vect{\alpha}^{(m)}\!\big(\mathcal{S}^{(n)}\big)\vect{1}_{d_{y_n}}  
}
\md x.
\end{align*} 
Taking dynamics in $u$ then yields the desired result. Finally, for $\bar{N}^{(m)}_{\mat{\iay}\mat{\jay}}$, $j\neq i$, we see that $N_{\mat{\iay}\mat{\jay}}$ can be written as
\begin{align*}
N_{\mat{\iay}\mat{\jay}}(u) = \sum_{n=1}^N\sum_{\ell = 1}^{M^{(n)}} \mathds{1}_{\big(Y_{\ell-1}^{(n)}\, =\, i,\, Y_{\ell}^{(n)}\, =\, j\big)}\mathds{1}_{\big(T_\ell^{(n)}\,\leq\, u\big)}\mathds{1}_{\big(\mat{X}^{(n)}(T_\ell^{(n)}-)\, =\, \mat{\iay}\big)}\mathds{1}_{\big(\mat{X}^{(n)}(T_\ell^{(n)})\, =\, \mat{\jay}\big)}
\end{align*}
such that
\begin{align*}
\bar{N}^{(m)}_{\mat{\iay}\mat{\jay}}(u) &= \sum_{n=1}^N\sum_{\ell = 1}^{M^{(n)}} \mathds{1}_{\big(Y_{\ell-1}^{(n)}\, =\, i,\, Y_{\ell}^{(n)}\, =\, j\big)}\mathds{1}_{\big(T_\ell^{(n)}\,\leq\, u\big)}\times  \\
&\qquad\qquad\qquad  \mathbb{E}^{(m)}\!\!\left[\left.\mathds{1}_{\big(\mat{X}^{(n)}(T_\ell^{(n)}-)\, =\, \mat{\iay}\big)}\mathds{1}_{\big(\mat{X}^{(n)}(T_\ell^{(n)})\, =\, \mat{\jay}\big)}\,
\right| \sigma\big(\mathcal{S}^{(n)}\big)\right]\!. 
\end{align*}
Now, since 
\begin{align*}
\mathbb{P}^{(m)}\!\!\left(\left.
\mat{X}^{(n)}(T_\ell^{(n)}) = \mat{\jay}\,
\right| \mat{X}^{(n)}(T_\ell^{(n)}-) = \mat{\iay}\right) = \mu_{\mat{\iay}\mat{\jay}}\big(T_\ell^{(n)}; \mat{\theta}^{(m)}\big), 
\end{align*}
we get, using the same techniques as in \eqref{eq:density}--\eqref{eq:lem_initial_dist}, that  
\begin{align*}
&\mathbb{P}\Big(
T^{(n)}_{M^{(n)}} \in \md t^{(n)}_{m^{(n)}}, Y^{(n)}_{M^{(n)}} = y^{(n)}_{m^{(n)}}, \ldots, 
T^{(n)}_{\ell}\in \md t^{(n)}_{\ell},  \mat{X}(T_{\ell}^{(n)}) = \mat{\jay}, \mat{X}(T_{\ell}^{(n)}-) = \mat{\iay},\\[0.2 cm]
&\quad\ T_{\ell-1}^{(n)}\in \md t^{(n)}_{\ell-1}, Y^{(n)}_{\ell-1} = i,\ldots, T_1 \in \md t_1, Y_1 = y_1\Big) \\[0.2 cm]
&\!\!\!\!\!  = \vect{\alpha}^{(m)}\!\big(\sai^{(n)}_{\ell-1}\big)\bar{\mat{P}}^{(m)}_i\!\big(t^{(n)}_{\ell-1},t^{(n)}_\ell\big)\vect{e}_{\tilde{i}}\mu_{\mat{\iay}\mat{\jay}}\big(t^{(n)}_\ell; \mat{\theta}^{(m)}\big)\vect{e}_{\tilde{j}}^{\prime}\mat{\bar{P}}_j^{(m)}\!\big(t^{(n)}_\ell,t^{(n)}_{\ell+1}\big)\vect{\alpha}^{(m)}_{\ell+2}\big(\sai^{(n)}\big)\!\md t^{(n)}_1\cdots\!\md t^{(n)}_{m^{(n)}}.
\end{align*}
Hence, 
\begin{align*}
\bar{N}^{(m)}_{\mat{\iay}\mat{\jay}}(u) &= \sum_{n=1}^N\sum_{\ell = 1}^{M^{(n)}} \mathds{1}_{\big(Y_{\ell-1}^{(n)}\, =\, i,\, Y_{\ell-1}^{(n)}\, =\, j\big)}\mathds{1}_{\big(T_\ell^{(n)}\,\leq\, u\big)}\times \\[0.2 cm]
&\frac{\vect{\alpha}^{(m)}\!\big(\mathcal{S}^{(n)}_{\ell-1}\big)\bar{\mat{P}}^{(m)}_i\! \big(T^{(n)}_{\ell-1},T^{(n)}_\ell\big)\vect{e}_{\tilde{i}}\mu_{\mat{\iay}\mat{\jay}}\big(T^{(n)}_\ell; \mat{\theta}^{(m)}\big)\vect{e}_{\tilde{j}}^{\prime}\mat{\bar{P}}^{(m)}_i\!\big(T^{(n)}_\ell,T^{(n)}_{\ell+1}\big)\vect{\alpha}^{(m)}_{\ell+2}\big(\mathcal{S}^{(n)}\big)}{\vect{\alpha}^{(m)}\big(\mathcal{S}^{(n)}\big)\vect{1}_{n} }. 
\end{align*}
Taking dynamics in $u$ then yields the desired result. 
\end{proof}



\begin{proof}[Proof of Theorem~\ref{cor:cond_log-likelihood_reset_general}]
It is evident from the complete data log-likelihood \eqref{eq:log-likelihood_reset} that the conditional expectation \eqref{eq:cond_log-likelihood_reset_general_def} is on the form \eqref{eq:cond_log-likelihood_reset_general}, by an application of Fubini's theorem.\ It, therefore, suffices to show that the conditional statistics are given as in \eqref{eq:cond_statistics_reset_general_1}--\eqref{eq:cond_statistics_reset_general_2}.\ To obtain these, the defining property is, as also noted in \cite[Subsection 4.2]{AhmadBladtFurrer}, that the reset property \eqref{eq:indep_cond} implies that for $j\neq i$, 
\begin{align*}
\mat{\pi}_i\big(t; \vect{\eta}^{(m)}\big)\bar{\mat{P}}^{(m)}(t,s)\vect{\beta}_{ij}\big(s; \mat{\theta}^{(m)}\big)
\end{align*}
is a $1\times 1$-dimensional matrix, and thus cancels if appearing in both the numerator and denominator of a fraction. Concerning the conditional statistics within macrostates, $\bar{I}^{(m)}_{\mat{\iay}}$ and $\bar{N}^{(m)}_{\mat{\iay}(i,\check{i})}$, this property implies
\begin{align}\label{eq:I/dN_reset}
\sum_{n=1}^N \frac{\mat{c}_i^{(m)}\!\big(u; \mathcal{S}^{(n)}\big)
}{
\vect{\alpha}^{(m)}\!\big(\mathcal{S}^{(n)}\big)\vect{1}_{n}
} = \sum_{n=1}^{N}\sum_{\ell = 1}^{M^{(n)}} 
\mathds{1}_{\big[T_{\ell-1}^{(n)},\, T_{\ell }^{(n)}\big)}\!(u)
\mathds{1}_{\big(Y_{\ell-1}^{(n)}\, =\, i\big)}
\frac{\mat{c}_i^{(m)}\!\big(u; T_{\ell-1}^{(n)}, i, T_{\ell}^{(n)}, Y_{\ell}^{(n)}\big)}
{
\vect{\alpha}^{(m)}\!\big(T_{\ell-1}^{(n)}, i, T_{\ell}^{(n)}, Y_{\ell}^{(n)}\big)\vect{1}_{n}
},
\end{align}
where, for generic values $0\leq x_1 < x_2 < \infty$, and $j\in \mathcal{J}$, $j\neq i$, we have for $u\in [x_1,x_2)$,   
\begin{align}\label{eq:ci_reset_general_mpp}
\begin{split}
\mat{c}_i^{(m)}(u; x_1, i, x_2, j) &=  \bar{\mat{P}}^{(m)}_{i}(u, x_2)
\vect{\beta}_{ij}\big(x_2; \mat{\theta}^{(m)}\big)    
\vect{\pi}_i\big(x_1; \vect{\eta}^{(m)}\big)
\bar{\mat{P}}^{(m)}_{i}(x_1, u),\\[0.2 cm]
\vect{\alpha}^{(m)}(x_1, i, x_2, j)\vect{1}_{d_j} &= \vect{\pi}_i\big(x_1; \vect{\eta}^{(m)}\big)
\bar{\mat{P}}^{(m)}_{i}(x_1, x_2)\vect{\beta}_{ij}\big(x_2; \mat{\theta}^{(m)}\big),
\end{split}
\end{align}
and zero otherwise. Hence, for a fixed time $u$ between two jump times, each term on the right-hand side of \eqref{eq:I/dN_reset} only depends on the last jump before $u$, the next jump time after $u$, and state jumped to at the next jump time; the past and future sojourns outside time $u$ are cancelled out. Consequently, when summing over all observations and corresponding sojourns, one can equivalently sum over the data points in $\mathcal{T}_i$, as these provide the jump times and states needed (cf.\ \eqref{eq:Ti_def}). This gives 
\begin{align}\label{eq:I/dN_reset_final}
\sum_{n=1}^N \frac{\mat{c}_i^{(m)}\!\big(u; \mathcal{S}^{(n)}\big)
}{
\vect{\alpha}^{(m)}\!\big(\mathcal{S}^{(n)}\big)\vect{1}_{n} 
} = \sum_{n=1}^{M_i}
\frac{\mat{c}_i^{(m)}\!\big(u; \mathcal{T}_{i}^{(n)}\big)}
{
\vect{\alpha}^{(m)}\!\big(\mathcal{T}_{i}^{(n)}\big)\vect{1}_{n}
},
\end{align}
which shows \eqref{eq:cond_statistics_reset_general_1} for $\bar{I}^{(m)}_{\mat{\iay}}$ and $\bar{N}^{(m)}_{\mat{\iay}(i,\check{i})}$.\ Concerning the conditional statistics for jumps between macrostates, $\bar{N}^{(m)}_{\mat{\iay}j}$, we use $\bar{N}^{(m)}_{\mat{\iay}\mat{\jay}}$ from Theorem \ref{thm:cond_log-lik_general} to get 
\begin{align*}
\bar{N}^{(m)}_{\mat{\iay}j}(u) &= 
\sum_{\tilde{j} = 1}^{d_j} \bar{N}_{\mat{\iay}\mat{\jay}}(u)\\[0.2 cm]
& = \sum_{n=1}^N \sum_{\ell = 1}^{M^{(n)}}\mathds{1}_{\big(T_{\ell}^{(n)}\, \geq\, u\big)}
\beta_{\mat{\iay}j}\big(T_\ell^{(n)}; \mat{\theta}^{(m)}\big)\vect{\pi}_{j}\big(T_\ell^{(n)}; \vect{\eta}^{(m)}\big)
\frac{  
\mat{a}_{ij}^{(m)}\!\big(T_\ell^{(n)}; \mathcal{S}^{(n)}\big)\vect{e}_{\tilde{i}}
}{
\vect{\alpha}^{(m)}\!\big(\mathcal{S}^{(n)}\big)\vect{1}_{n}
}.
\end{align*}
Using the same technique with the fraction as in \eqref{eq:I/dN_reset}--\eqref{eq:I/dN_reset_final}, we get 
\begin{align*}
&\bar{N}^{(m)}_{\mat{\iay}j}(u)\\[0.2 cm] 
&= \sum_{n=1}^N \sum_{\ell = 1}^{M^{(n)}}
\mathds{1}_{\big(T_{\ell}^{(n)}\, \geq\, u\big)}\!\mathds{1}_{\big(Y_{\ell-1}^{(n)}\, =\, i, \, Y_{\ell}^{(n)}\, =\, j \big)}
\frac{  
\vect{\pi}_{i}\big(T_{\ell-1}^{(n)}; \vect{\eta}^{(m)}\big)
\bar{\mat{P}}^{(m)}_i\!\big(T_{\ell-1}^{(n)}, T_{\ell}^{(n)} \big)
\vect{e}_{\tilde{i}}\beta_{\mat{\iay}j}\big(T_\ell^{(n)}; \mat{\theta}^{(m)}\big)
}{
\vect{\alpha}^{(m)}\!\big(T_{\ell-1}^{(n)}, i, T_{\ell}^{(n)}, j\big)\vect{1}_{n}
} \\[0.2 cm]
&= \sum_{n=1}^{M_i} \mathds{1}_{\big(\tau_{i}^{(n)}\, \geq\, u\big)}\mathds{1}_{\big(Z_{i}^{(n)}\, =\, j \big)}
\frac{  
\vect{\pi}_{i}\big(R_{i}^{(n)}; \vect{\eta}^{(m)}\big)
\bar{\mat{P}}^{(m)}_i\!\big(R_{i}^{(n)}, \tau_{i}^{(n)} \big)
\vect{e}_{\tilde{i}}\beta_{\mat{\iay}j}\big(\tau_i^{(n)}; \mat{\theta}^{(m)}\big)
}{
\vect{\alpha}^{(m)}\!\big(\mathcal{T}_i^{(n)}\big)\vect{1}_{n}
}.
\end{align*}
Taking dynamics in $u$ then yields \eqref{eq:cond_statistics_reset_general_2}. Concerning the conditional statistics for initiations into macrostates, $\bar{N}^{(m)}_{\mat{\iay}}$, we use the exact same approach as for $\bar{N}_{\mat{\iay}j}^{(m)}$, with the only difference being the summation over $\mat{\jay}$ of $\vect{e}_{\tilde{j}}\beta_{\mat{\jay}i}(\cdot; \mat{\theta}^{(m)})$ instead of a summation over $\tilde{j}$ of $\pi_{\mat{\jay}}(\cdot; \vect{\eta}^{(m)})\vect{e}_{\tilde{j}}^{\prime}$. 
\end{proof}

\bibliographystyle{plain}
\bibliography{EstimationMain.bbl}{}

\end{document}